\documentclass[10pt]{amsart}

\usepackage{amsfonts,amssymb,amsmath,amsthm,latexsym,graphics,epsfig}
\usepackage{verbatim,enumerate,array,booktabs,color,bigstrut,prettyref,tikz-cd}
\usepackage{multirow}
\usepackage[all]{xy}
\usepackage[backref]{hyperref}
\usepackage[OT2,T1]{fontenc}
\usepackage{ctable}

\newrefformat{eq}{\textup{(\ref{#1})}}
\newrefformat{prty}{\textup{(\ref{#1})}}

\definecolor{mylinkcolor}{rgb}{0.8,0,0}
\definecolor{myurlcolor}{rgb}{0,0,0.8}
\definecolor{mycitecolor}{rgb}{0,0,0.8}
\hypersetup{colorlinks=true,urlcolor=myurlcolor,citecolor=mycitecolor,linkcolor=mylinkcolor,linktoc=page,breaklinks=true}

\DeclareSymbolFont{cyrletters}{OT2}{wncyr}{m}{n}
\DeclareMathSymbol{\Sha}{\mathalpha}{cyrletters}{"58}

\newtheorem{defn}{Definition}[section]
\newtheorem{definition}[defn]{Definition}

\newtheorem{corollary}[defn]{Corollary}
\newtheorem{lemma}[defn]{Lemma}

\newtheorem{thm}[defn]{Theorem}
\newtheorem{theorem}[defn]{Theorem}

\newtheorem{proposition}[defn]{Proposition}

\theoremstyle{definition}

\newtheorem*{ack}{Acknowledgements}
\newtheorem{remark}[defn]{Remark}

\newtheorem{example}[defn]{Example}

\newcommand{\QQ}{\mathbb Q}

\newcommand{\ZZ}{\mathbb Z}
\newcommand{\Z}{\mathbb Z}
\newcommand{\Q}{\mathbb{Q}}

\newcommand{\arrow}{\longrightarrow}

\newcommand{\tor}{\mathrm{tors}}

\begin{document}
	
	% Title, authors and addresses
	
	% use the thanksref command within \title, \author or \address for footnotes;
	% use the corauthref command within \author for corresponding author footnotes;
	% use the ead command for the email address,
	% and the form \ead[url] for the home page:
	% \title{Title\thanksref{label1}}
	% \thanks[label1]{}
	% \author{Name\corauthref{cor1}\thanksref{label2}}
	% \ead{email address}
	% \ead[url]{home page}
	% \thanks[label2]{}
	% \corauth[cor1]{}
	% \address{Address\thanksref{label3}}
	% \thanks[label3]{}
	
	\title[2-adic Galois images of CM isogeny-torsion graphs]{2-adic Galois images of isogeny-torsion graphs over $\QQ$ with CM}
	
\author{Garen Chiloyan}
\email{garen.chiloyan@gmail.com} 
\urladdr{https://sites.google.com/view/garenmath/home}

%\author{J. M. Tornero}

%\keywords{modular curves}

%\thanks{The first author was partially  supported by the grant MTM2012--35849.}
% Authors and running title to go on top of each page
%\pagestyle{myheadings} \markboth{\'Alvaro Lozano-Robledo}{Rank
%over large fields.}

\maketitle

%\part{Use this type of header for very long papers only}

%%%%%%%%%%%%%%%%%%%%%%%%%%%%%%%%%%%%%%%%%%%%%%%%%%%%%%%%%%%%%%%%%%%%%%%%%%%%%%%%
%%%%%%%%%%%%%%%%%%%%%%%%%%%%%%%%%%%%%%%%%%%%%%%%%%%%%%%%%%%%%%%%%%%%%%%%%%%%%%%%
%%%%%%%%%%%%%%%%%%%%%%%%%%%%%%%%%%%%%%%%%%%%%%%%%%%%%%%%%%%%%%%%%%%%%%%%%%%%%%%%

\begin{abstract}

Let $\mathcal{E}$ be a $\QQ$-isogeny class of elliptic curves defined over $\QQ$. The isogeny graph associated to $\mathcal{E}$ is a graph which has a vertex for each elliptic curve over $\QQ$ of $\mathcal{E}$ and an edge for each $\QQ$-isogeny of prime degree that maps one elliptic curve in $\mathcal{E}$ to another elliptic curve in $\mathcal{E}$, with the degree of the isogeny recorded as a label of the edge. The isogeny-torsion graph associated to $\mathcal{E}$ is the isogeny graph associated to $\mathcal{E}$ where, in addition, we label each vertex with the abstract group structure of the torsion subgroup over $\QQ$ of the corresponding elliptic curve. The main result of the article is a classification of the $2$-adic Galois image at each vertex of the isogeny-torsion graphs whose associated $\QQ$-isogeny class consists of elliptic curves over $\QQ$ with complex multiplication.

\end{abstract}

\section{Introduction}

Let $E/\QQ$ be an elliptic curve. It is well known that $E$ has the structure of an abelian group with group identity which we will denote $\mathcal{O}$. By the Mordell--Weil theorem, the set of points on $E$ defined over $\QQ$, denoted $E(\QQ)$ has the structure of a finitely generated abelian group. Thus, the set of points on $E$ defined over $\QQ$ of finite order, denoted $E(\QQ)_{\text{tors}}$ is a finite, abelian group. By Mazur's theorem, $E(\QQ)_{\texttt{tors}}$ is isomorphic to one of fifteen groups (see Theorem \ref{thm-mazur}). Moreover, these fifteen groups occur infinitely often. Let $E'/\QQ$ be an elliptic curve. An isogeny mapping $E$ to $E'$ is a rational morphism $\phi \colon E \to E'$ such that $\phi$ maps the identity of $E$ to the identity of $E'$. If there is a non-constant isogeny defined over $\QQ$, mapping $E$ to $E'$, we say that $E$ is $\QQ$-isogenous to $E'$. This relation is an equivalence relation and the set of elliptic curves defined over $\QQ$ that are $\QQ$-isogenous to $E$ is called the $\QQ$-isogeny class of $E$.

An isogeny is a group homomorphism and the kernel of a non-constant isogeny is finite. We are particularly interested in non-constant isogenies with cyclic kernels. The isogeny graph associated to the $\QQ$-isogeny class of $E$ is a visual description of the $\QQ$-isogeny class of $E$. Denote the $\QQ$-isogeny class of $E$ by $\mathcal{E}$. The isogeny graph associated to $\mathcal{E}$ is a graph which has a vertex for each elliptic curve in $\mathcal{E}$ and an edge for each $\QQ$-isogeny of prime degree that maps one elliptic curve in $\mathcal{E}$ to another elliptic curve in $\mathcal{E}$, with the degree recorded as a label of the edge. The isogeny-torsion graph associated to $\mathcal{E}$ is the isogeny graph associated to $\mathcal{E}$ where, in addition, we label each vertex with the abstract group structure of the torsion subgroup over $\Q$ of the corresponding elliptic curve.

\begin{example}\label{T4 example}

There are four elliptic curves in the $\QQ$-isogeny class with LMFDB label \texttt{27.a} which we will denote $E_{1}$, $E_{2}$, $E_{3}$, and $E_{4}$. The isogeny graph associated to \texttt{27.a} is above and the isogeny-torsion graph associated to \texttt{27.a} is below.

\begin{center} \begin{tikzcd}
E_{1} \arrow[r, "3", no head] & E_{2} \arrow[r, "3", no head] & E_{3} \arrow[r, "3", no head] & E_{4}
\end{tikzcd} \end{center}

\begin{center} \begin{tikzcd}
\mathbb{Z} / 3 \mathbb{Z} \arrow[r, no head, "3"] & \mathbb{Z} / 3 \mathbb{Z} \arrow[r, no head, "3"] & \mathbb{Z} / 3 \mathbb{Z} \arrow[r, no head, "3"] & \mathcal{O}
\end{tikzcd}
\end{center}
\end{example}

A proof classifying the isogeny graphs associated to $\QQ$-isogeny classes of elliptic curves over $\QQ$ appears in Section 6 of \cite{gcal-r}.

\begin{thm}
	There are $26$ isomorphism types of isogeny graphs that are associated to $\QQ$-isogeny classes of elliptic curves defined over $\QQ$. More precisely, there are $16$ types of (linear) $L_k$ graphs of $k = 1$-$4$ vertices, $3$ types of (nonlinear two-primary torsion) $T_k$ graphs of $k = 4$, $6$, or $8$ vertices, $6$ types of (rectangular) $R_k$ graphs of $k = 4$ or $6$ vertices, and $1$ (special) $S$ graph (see Tables 1 - 4 in \cite{gcal-r}).
\end{thm}

The isogeny class degree of $\mathcal{E}$ is the least common multiple of the degrees of all cyclic, $\QQ$-rational isogenies mapping elliptic curves over $\QQ$ in $\mathcal{E}$ to elliptic curves over $\QQ$ in $\mathcal{E}$. In other words, the isogeny class degree of $\mathcal{E}$ is equal to the greatest degree of a cyclic, $\QQ$-rational isogeny that maps an elliptic curve in $\mathcal{E}$ to an elliptic curve in $\mathcal{E}$. For example, if $\mathcal{E}$ is of $L_{4}$ type, the isogeny class of $\mathcal{E}$ is equal to $27$.

In the case of an isogeny graph of $L_{2}$, $L_{3}$, or $R_{4}$ type, the isogeny class degree of the $\QQ$-isogeny class is written in parentheses to distinguish it from other isogeny-torsion graphs of the same size and shape, but with different isogeny class degree. For example, there are $L_{2}(2)$ graphs; graphs of $L_{2}$ type generated by an isogeny of degree $2$ and there are $L_{2}(3)$ graphs; isogeny graphs of $L_{2}$ type generated by an isogeny of degree $3$. Relying only on the size and shape of isogeny graphs of $L_{2}$ type is not enough to distinguish isogeny graphs of $L_{2}(2)$ type from isogeny graphs of $L_{2}(3)$ type. On the other hand, the isogeny-torsion graph of $L_{4}$ type is the only linear isogeny-torsion graph with four vertices and it is not necessary to designate the isogeny class degree in parentheses to distinguish it from other isogeny-torsion graphs. The main theorem in \cite{gcal-r} was the classification of isogeny-torsion graphs associated to $\QQ$-isogeny classes of elliptic curves over $\QQ$.

\begin{thm}[Chiloyan, Lozano-Robledo, \cite{gcal-r}]
There are $52$ isomorphism types of isogeny-torsion graphs that are associated to $\Q$-isogeny classes of elliptic curves defined over $\Q$. In particular, there are $23$ isogeny-torsion graphs of $L_k$ type, $13$ isogeny-torsion graphs of $T_k$ type, $12$ isogeny-torsion graphs of $R_k$ type, and $4$ isogeny-torsion graphs of $S$ type.
\end{thm}

We denote the cyclic group of order $a$ as $[a]$ and for $b = 1 - 4$, we denote the group $\ZZ / 2 \ZZ \times \ZZ / 2 \cdot b \ZZ$ as $[2,b]$. We organize torsion configuration of isogeny-torsion graphs in ``vector-group'' formation corresponding to the enumeration of the elliptic curves in the isogeny-torsion graph. For example, reconsider the $\QQ$-isogeny class with LMFDB label \texttt{27.a} and its associated isogeny graph and isogeny-torsion graph, $\mathcal{G}$.

\begin{center} \begin{tikzcd}
E_{1} \arrow[r, "3", no head] & E_{2} \arrow[r, "3", no head] & E_{3} \arrow[r, "3", no head] & E_{4}
\end{tikzcd} \end{center}

\begin{center} \begin{tikzcd}
\mathbb{Z} / 3 \mathbb{Z} \arrow[r, no head, "3"] & \mathbb{Z} / 3 \mathbb{Z} \arrow[r, no head, "3"] & \mathbb{Z} / 3 \mathbb{Z} \arrow[r, no head, "3"] & \mathcal{O}
\end{tikzcd}
\end{center}
Then we will denote the torsion configuration of $\mathcal{G}$ as $([3],[3],[3],[1])$. For another case, consider the $\QQ$-isogeny class $\mathcal{E}$ with LMFDB label \texttt{17.a}. Then the isogeny graph of $\mathcal{E}$ is below on the left and the isogeny-torsion graph of $\mathcal{E}$ is below on the right
\begin{center}
    \begin{tikzcd}
      & E_{2}                                                 &       \\
      & E_{1} \arrow[u, no head, "2"] \arrow[ld, no head, "2"'] \arrow[rd, no head, "2"] &       \\
E_{3} &                                                       & E_{4}
\end{tikzcd}, \begin{tikzcd}
                          & \mathbb{Z} / 4 \mathbb{Z}                                                                                  &                           \\
                          & \mathbb{Z} / 2 \mathbb{Z} \times \mathbb{Z} / 2 \mathbb{Z} \arrow[u, no head, "2"] \arrow[ld, no head, "2"'] \arrow[rd, no head, "2"] &                           \\
\mathbb{Z} / 4 \mathbb{Z} &                                                                                                            & \mathbb{Z} / 2 \mathbb{Z}
\end{tikzcd}
\end{center}
We denote the torsion classification of $\mathcal{G}$ to be $([2,2],[4],[4],[2])$.

Let $E/\QQ$ be an elliptic curve with complex multiplication (CM) and denote the $\QQ$-isogeny class of $E$ by $\mathcal{E}$ and the isogeny-torsion graph associated to $\mathcal{E}$ by $\mathcal{G}$. Then all elliptic curves over $\QQ$ in $\mathcal{E}$ have CM. Thus, it is natural to say that $\mathcal{E}$ and $\mathcal{G}$ are defined over $\QQ$ and have CM. Moreover, it is natural to say that the classification of the $2$-adic Galois image of each of the vertices of $\mathcal{G}$ would classify the $2$-adic Galois image of $\mathcal{G}$. The isogeny-torsion graph, $\mathcal{G}$ is of $L_{2}(p)$, $L_{4}$, $T_{4}$, $R_{4}(6)$, or $R_{4}(14)$ type, where $p \in \{2, 3, 11, 19, 43, 67, 163\}$. Table \ref{tab-CMgraphs} classifies the isogeny-torsion graphs with CM. This same table appears as Table 5 in Section 4 of \cite{gcal-r}.
\newpage

\begin{table}[h!]
 	\renewcommand{\arraystretch}{1.2}
 	\begin{tabular}{|c|c|c|c|c|c|}
 		\hline
 		$d_K$ & \multicolumn{2}{c|}{$j$} & Type & Torsion config. & LMFDB\\
 		\hline
 		\hline
 		\multirow{10}{*}{$-3$}  & \multirow{6}{*}{$0$} & $y^2=x^3+t^3, t=-3,1$  & $R_4(6)$ & $([6],[6],[2],[2])$ & \texttt{36.a4}\\
 		& &$y^2=x^3+t^3, t\neq -3,1$  & $R_4(6)$ & $([2],[2],[2],[2])$ & \texttt{144.a3}\\
 		& & $y^2=x^3+16t^3, t=-3,1$ & $L_4$ & $([3],[3],[3],[1])$ & \texttt{27.a3}\\
 		& & $y^2=x^3+16t^3, t\neq -3,1$  & $L_4$ & $([1],[1],[1],[1])$ & \texttt{432.e3}\\
 		& & $\!y^2=x^3+s^2,\,s^2\neq t^3,16t^3\!\!$  & $L_2(3)$ & $([3],[1])$ & \texttt{108.a2}\\
 		& & $\!y^2=x^3+s,\,s\neq t^3,16t^3\!\!$  & $L_2(3)$ & $([1],[1])$ & \texttt{225.c1}\\
 		\cline{2-6}
 		& \multirow{2}{*}{$54000$} &  $y^2=x^3-15t^2x + 22t^3, t=1,3$ & $R_4(6)$ & $([6],[6],[2],[2])$ & \texttt{36.a1}\\
 		&  & $y^2=x^3-15t^2x + 22t^3, t\neq 1,3$ & $R_4(6)$ & $([2],[2],[2],[2])$ & \texttt{144.a1}\\
 		\cline{2-6}
 		& \multirow{2}{*}{$-12288000$} & $E^t, t=-3,1$ & $L_4$ & $([3],[3],[3],[1])$ & \texttt{27.a2}\\
 		& & $E^t, t\neq -3,1$ & $L_4$ & $([1],[1],[1],[1])$ & \texttt{432.e1}\\
 		\hline
 		\multirow{7}{*}{$-4$} & \multirow{4}{*}{$1728$} & $y^2=x^3+tx, t=-1,4$  & $T_4$ & $([2,2],[4],[4],[2])$ & \texttt{32.a3}\\
 		& & $y^2=x^3+tx, t=-4,1$  & $T_4$ & $([2,2],[4],[2],[2])$ & \texttt{64.a3}\\
 		& & $y^2=x^3\pm t^2x, t\neq 1,2$  & $T_4$ & $([2,2],[2],[2],[2])$ & \texttt{288.d3}\\
 		& &  $y^2=x^3+sx$, $s\neq \pm t^2$  & $L_2(2)$ & $([2],[2])$ & \texttt{256.b1}\\
 		\cline{2-6}
 		& \multirow{3}{*}{$287496$} & $y^2=x^3-11t^2x+14t^3, t=\pm 1$ & \multirow{3}{*}{$T_4$} & $([2,2],[4],[4],[2])$ & \texttt{32.a2}\\
 		& & $y^2=x^3-11t^2x+14t^3, t=\pm 2$ &  & $([2,2],[4],[2],[2])$ & \texttt{64.a1}\\
 		& & $y^2=x^3-11t^2x+14t^3, t\neq \pm 1, \pm 2$ &  & $([2,2],[2],[2],[2])$ & \texttt{288.d1}\\
 		\hline
 		\multirow{2}{*}{$-7$} & \multicolumn{2}{c|}{$-3375$}  & \multirow{2}*{$R_4(14)$} & $([2],[2],[2],[2])$ & \texttt{49.a2}\\
 		& \multicolumn{2}{c|}{$16581375$} & & $([2],[2],[2],[2])$  & \texttt{49.a1}\\
 		\hline
 		$-8$ & \multicolumn{2}{c|}{$8000$}  & $L_2(2)$ & $([2],[2])$ & \texttt{256.a1}\\
 		\hline
 		$-11$ & \multicolumn{2}{c|}{$-32768$} & $L_2(11)$ & $([1],[1])$ & \texttt{121.b1}\\
 		\hline
 		$-19$ & \multicolumn{2}{c|}{$-884736$} & $L_2(19)$ & $([1],[1])$ & \texttt{361.a1} \\
 		\hline
 		$-43$ & \multicolumn{2}{c|}{$-884736000$}  & $L_2(43)$ & $([1],[1])$ & \texttt{1849.b1}\\
 		\hline
 		$-67$ & \multicolumn{2}{c|}{$-147197952000$}  & $L_2(67)$ & $([1],[1])$ & \texttt{4489.b1} \\
 		\hline
 		$-163$ & \multicolumn{2}{c|}{$-262537412640768000$} & $L_2(163)$ & $([1],[1])$ & \texttt{26569.a1} \\
 		\hline
 	\end{tabular}
 	\caption{The list of rational $j$-invariants with CM and the possible isogeny-torsion graphs that occur, where $E^t$ denotes the curve $y^2=x^3-38880t^2x+2950992t^3$.}
 	\label{tab-CMgraphs}
 \end{table}

\newpage

\begin{example}
Let $E/\QQ$ be an elliptic curve such that the isogeny graph associated to the $\QQ$-isogeny class of $E$ is of $L_{4}$ type. Then $E$ is represented by one of the elliptic curves $E_{1}$, $E_{2}$, $E_{3}$, or $E_{4}$ in the isogeny graph below.
\begin{center} \begin{tikzcd}
E_{1} \arrow[r, "3", no head] & E_{2} \arrow[r, "3", no head] & E_{3} \arrow[r, "3", no head] & E_{4}
\end{tikzcd} \end{center}
No matter the torsion configuration of the isogeny-torsion graph, $\rho_{E,2^{\infty}}(G_{\QQ})$ is conjugate to the group
$$\left\langle \operatorname{-Id}, \begin{bmatrix} 0 & 1 \\ 1 & 0 \end{bmatrix}, \begin{bmatrix} 7 & 4 \\ -4 & 3 \end{bmatrix}, \begin{bmatrix} 3 & 6 \\ -6 & -3 \end{bmatrix} \right\rangle \subseteq \operatorname{GL}(2, \ZZ_{2}).$$
\end{example}
The reason why the $2$-adic Galois image of all elliptic curves in an $L_{4}$ graph are conjugate comes from the fact that $3$ is odd (see Corollary \ref{coprime isogeny-degree}). The situation will not always be as seamless as this. When there are non-trivial, cyclic isogenies of $2$-power degree in the isogeny graph, it is likely the $2$-adic Galois image of the vertices are different.

Let $\delta$, $\phi$, and $N$ be integers such that $N \geq 0$. Denote the subgroup of $\operatorname{GL}\left(2, \ZZ / 2^{N} \ZZ\right)$ of matrices of the form $\begin{bmatrix}
a + b \cdot \phi & b \\ \delta \cdot b & a
\end{bmatrix}$ by $\mathcal{C}_{\delta, \phi}\left(2^{N}\right)$ and let $\mathcal{N}_{\delta, \phi}\left(2^{N}\right) = \left\langle \mathcal{C}_{\delta, \phi}(2^{N}), \begin{bmatrix} -1 & 0 \\ \phi & 1 \end{bmatrix} \right\rangle$. Finally, let $\mathcal{N}_{\delta,\phi}\left(2^{\infty}\right) = \varprojlim N_{\delta,\phi}(N)$.

\begin{example}
Let $E/\QQ$ be an elliptic curve such that the isogeny graph associated to the $\QQ$-isogeny class of $E$ is of $R_{4}(14)$ type (see below).
\begin{center}
\begin{tikzcd}
{E_{1}} \arrow[dd, no head, "7"'] \arrow[rr, no head, "2"] &  & {E_{2}} \arrow[dd, no head, "7"] \\
                                                                     &  &                                                        \\
{E_{3}} \arrow[rr, no head, "2"']                 &  & {E_{4}}               
\end{tikzcd}
\end{center}
If $E$ is represented by $E_{1}$ or $E_{3}$, then $\rho_{E,2^{\infty}}(G_{\QQ})$ is conjugate to $\mathcal{N}_{-7,0}(2^{\infty})$ and if $E$ is represented by $E_{2}$ or $E_{4}$, then $\rho_{E,2^{\infty}}(G_{\QQ})$ is conjugate to $\mathcal{N}_{-7,1}(2^{\infty})$.

\iffalse
Then $\rho_{E,2^{\infty}}(G_{\QQ})$ is conjugate to $\mathcal{N}_{-7,0}(2^{\infty})  = \left\langle C_{-7,0}(2^{\infty}), \begin{bmatrix} -1 & 0 \\ 0 & 1 \end{bmatrix} \right\rangle$ or $\rho_{E,2^{\infty}}(G_{\QQ})$ is conjugate to $\mathcal{N}_{-7,1}(2^{\infty}) = \left\langle C_{-7,0}(2^{\infty}), \begin{bmatrix} -1 & 0 \\ 1 & 1 \end{bmatrix} \right\rangle$ where $C_{-7,0}(2^{\infty})$ is the set of all $2 \times 2$, invertible matrices over $\ZZ_{2}$ of the form $\begin{bmatrix} a & b \\ -7b & a \end{bmatrix}$.
\fi
\end{example}

Section \ref{sec-background} will be devoted to going over background and some lemmas, Section \ref{sec-Alvaros work} will be devoted to going over work by Lozano-Robledo in classifying the $2$-adic Galois image of elliptic curves defined over $\QQ$ with complex multiplication and Section \ref{proofs} will have the proof of Proposition \ref{propmain} and will fully classify the $2$-adic Galois image attached to isogeny-torsion graphs defined over $\QQ$ with complex multiplication. The proof will be broken up into many steps, appealing to isogeny graphs and \textit{j}-invariants.

\begin{proposition}\label{propmain}
Let $\mathcal{G}$ be a CM isogeny-torsion graph defined over $\QQ$. Then $\mathcal{G}$ fits into Table \ref{TableA} or Table \ref{TableB} with the given classification of the corresponding $2$-adic Galois image of its vertices. Examples of the possible CM isogeny-torsion graphs with the given $2$-adic Galois image classification are provided in the final column of each table.
\newpage
\begin{center}
\begin{table}[h!]
\renewcommand{\arraystretch}{1.3}
\scalebox{0.49}{
    \begin{tabular}{|c|c|c|c|c|c|c|}
    \hline
        Isogeny Graph & Torsion & $\rho_{E_{1},2^{\infty}}(G_{\QQ})$ & $\rho_{E_{2},2^{\infty}}(G_{\QQ})$ & $\rho_{E_{3},2^{\infty}}(G_{\QQ})$ & $\rho_{E_{4},2^{\infty}}(G_{\QQ})$ & Example\\
         \hline
        \multirow{2}*{\includegraphics[width=40mm]{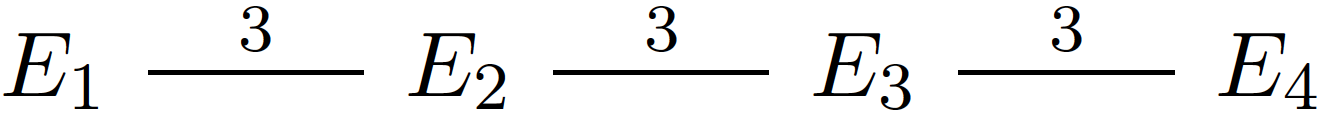}} & $([3],[3],[3],[1])$ & \multirow{2}*{$\mathcal{N}_{-1,1}(2^{\infty})$} & \multirow{2}*{$\mathcal{N}_{-1,1}(2^{\infty})$} & \multirow{2}*{$\mathcal{N}_{-1,1}(2^{\infty})$} & \multirow{2}*{$\mathcal{N}_{-1,1}(2^{\infty})$} & \texttt{27.a} \\
        \cline{2-2}
        \cline{7-7}
        & $([1],[1],[1],[1])$ & & & & & \texttt{432.e} \\
        \cline{1-7}
        
        \multirow{4}*{\includegraphics[width=20mm]{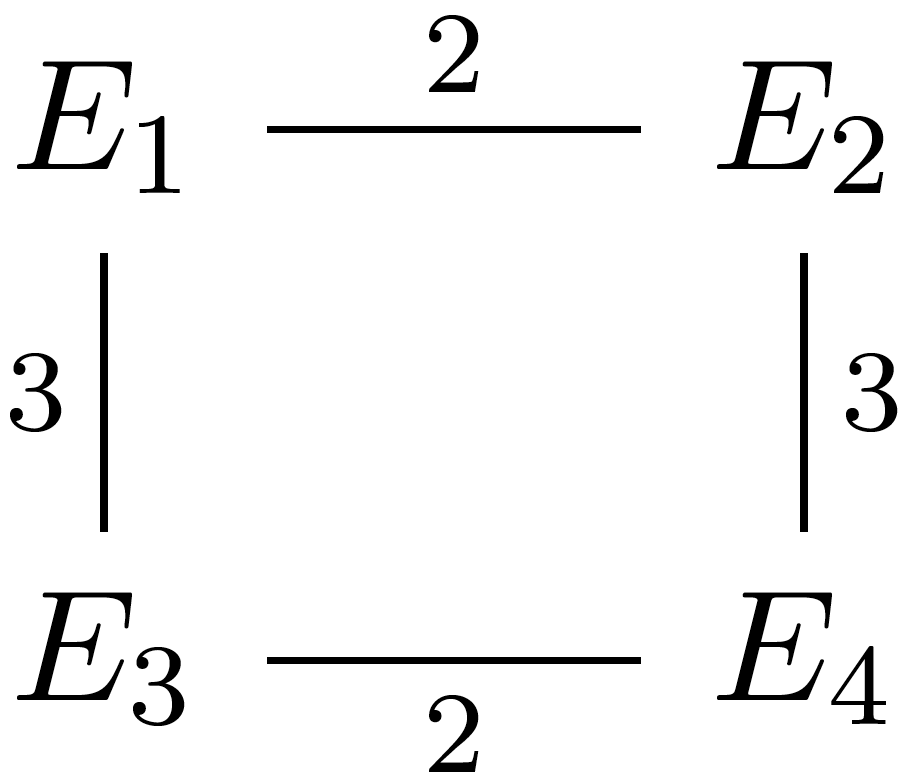}} & \multirow{2}*{$([6],[6],[2],[2])$} & \multirow{4}*{$\left\langle \operatorname{-Id}, \begin{bmatrix} 0 & 1 \\ 1 & 0 \end{bmatrix}, \begin{bmatrix} 7 & 4 \\ -4 & 3 \end{bmatrix}, \begin{bmatrix} 3 & 6 \\ -6 & -3 \end{bmatrix} \right\rangle$} & \multirow{4}*{$\mathcal{N}_{-3,0}(2^{\infty})$} & \multirow{4}*{$\left\langle \operatorname{-Id}, \begin{bmatrix} 0 & 1 \\ 1 & 0 \end{bmatrix}, \begin{bmatrix} 7 & 4 \\ -4 & 3 \end{bmatrix}, \begin{bmatrix} 3 & 6 \\ -6 & -3 \end{bmatrix} \right\rangle$} & \multirow{4}*{$\mathcal{N}_{-3,0}(2^{\infty})$} & \multirow{2}*{\texttt{36.a}} \\
        & & & & & & \\
        \cline{2-2}
        \cline{7-7}
        & \multirow{2}*{$([2],[2],[2],[2])$} & & & & & \multirow{2}*{\texttt{144.a}} \\
        & & & & & & \\
        \cline{1-7}
        
        \multirow{4}*{\includegraphics[width=20mm]{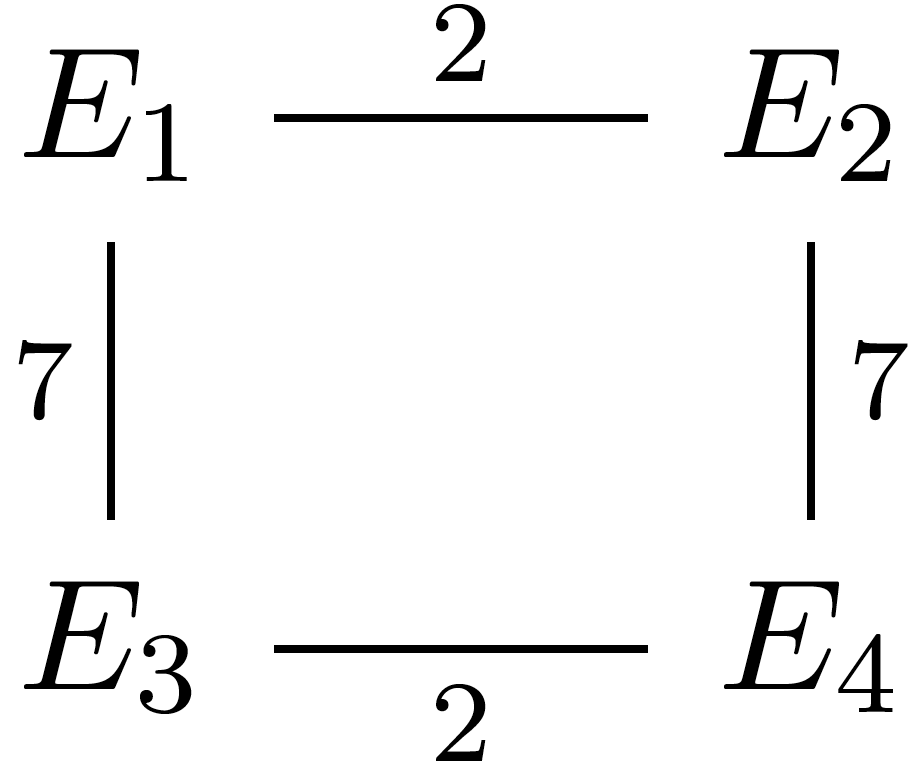}} & \multirow{4}*{$([2],[2],[2],[2])$} & \multirow{4}*{$\mathcal{N}_{-7,0}(2^{\infty})$} & \multirow{4}*{$\mathcal{N}_{-2,1}(2^{\infty})$} & \multirow{4}*{$\mathcal{N}_{-7,0}(2^{\infty})$} & \multirow{4}*{$\mathcal{N}_{-2,1}(2^{\infty})$} & \multirow{4}*{\texttt{49.a}} \\
        & & & & & & \\
        & & & & & & \\
        & & & & & & \\
        \hline
        
        \multirow{6}*{\includegraphics[width = 25mm]{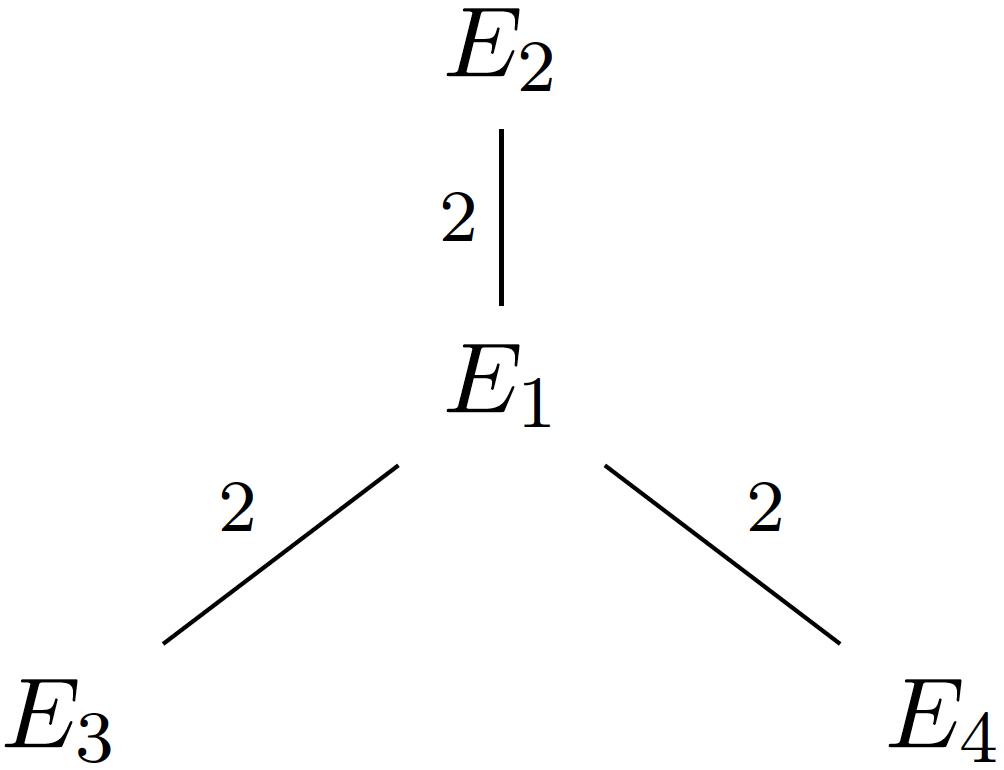}} & \multirow{2}*{$([2,2],[4],[4],[2])$} & \multirow{2}*{$\left\langle 5 \cdot \operatorname{Id} \begin{bmatrix} -1 & -2 \\ 2 & -1 \end{bmatrix}, \begin{bmatrix} 1 & 0 \\ 0 & -1 \end{bmatrix} \right\rangle$} & \multirow{2}*{$\left\langle 5 \cdot \operatorname{Id} \begin{bmatrix} -1 & -2 \\ 2 & -1 \end{bmatrix}, \begin{bmatrix} 0 & 1 \\ 1 & 0 \end{bmatrix} \right\rangle$} & \multirow{2}*{$\left\langle 5 \cdot \operatorname{Id}, \begin{bmatrix} 1 & 0 \\ 0 & -1 \end{bmatrix}, \begin{bmatrix} -1 & -1 \\ 4 & -1 \end{bmatrix} \right\rangle$} & \multirow{2}*{$\left\langle 5 \cdot \operatorname{Id}, \begin{bmatrix} -1 & 0 \\ 0 & 1 \end{bmatrix}, \begin{bmatrix} -1 & -1 \\ 4 & -1 \end{bmatrix} \right\rangle$} & \multirow{2}*{\texttt{32.a}} \\
        & & & & & & \\
        \cline{2-7}
        
        & \multirow{2}*{$([2,2],[2],[4],[2])$} & \multirow{2}*{$\left\langle 5 \cdot \operatorname{Id} \begin{bmatrix} 1 & 2 \\ -2 & 1 \end{bmatrix}, \begin{bmatrix} 1 & 0 \\ 0 & -1 \end{bmatrix} \right\rangle$} & \multirow{2}*{$\left\langle 5 \cdot \operatorname{Id} \begin{bmatrix} 1 & 2 \\ -2 & 1 \end{bmatrix}, \begin{bmatrix} 0 & 1 \\ 1 & 0 \end{bmatrix} \right\rangle$} & \multirow{2}*{$\left\langle 5 \cdot \operatorname{Id}, \begin{bmatrix} 1 & 0 \\ 0 & -1 \end{bmatrix}, \begin{bmatrix} 1 & 1 \\ -4 & 1 \end{bmatrix} \right\rangle$} & \multirow{2}*{$\left\langle 5 \cdot \operatorname{Id}, \begin{bmatrix} -1 & 0 \\ 0 & 1 \end{bmatrix}, \begin{bmatrix} 1 & 1 \\ -4 & 1 \end{bmatrix} \right\rangle$} & \multirow{2}*{\texttt{64.a}} \\
        & & & & & & \\
        \cline{2-7}
        
        & \multirow{2}*{$([2,2],[2],[2],[2])$} & \multirow{2}*{$\left\langle \operatorname{-Id}, 3 \cdot \operatorname{Id}, \begin{bmatrix} 1 & 2 \\ -2 & 1 \end{bmatrix}, \begin{bmatrix} 1 & 0 \\ 0 & -1 \end{bmatrix} \right\rangle$} & \multirow{2}*{$\left\langle \operatorname{-Id}, 3 \cdot \operatorname{Id}, \begin{bmatrix} 1 & 2 \\ -2 & 1 \end{bmatrix}, \begin{bmatrix} 0 & 1 \\ 1 & 0 \end{bmatrix} \right\rangle$} & \multirow{2}*{$\mathcal{N}_{-4,0}(2^{\infty})$} & \multirow{2}*{$\mathcal{N}_{-4,0}(2^{\infty})$} & \multirow{2}*{288.d} \\
        & & & & & & \\
        \hline

\end{tabular}}
	\caption{}
 	\label{TableA}
\end{table}
\end{center}

\begin{center}
\begin{table}[h!]
\renewcommand{\arraystretch}{1.3}
\scalebox{0.79}{
    \begin{tabular}{|c|c|c|c|c|c|}
    \hline
    Isogeny Graph & $p$ & Torsion & $\rho_{E_{1},2^{\infty}}(G_{\QQ})$ & $\rho_{E_{2},2^{\infty}}(G_{\QQ})$ & Example \\
    \hline
    \multirow{26}*{\includegraphics[width = 40mm]{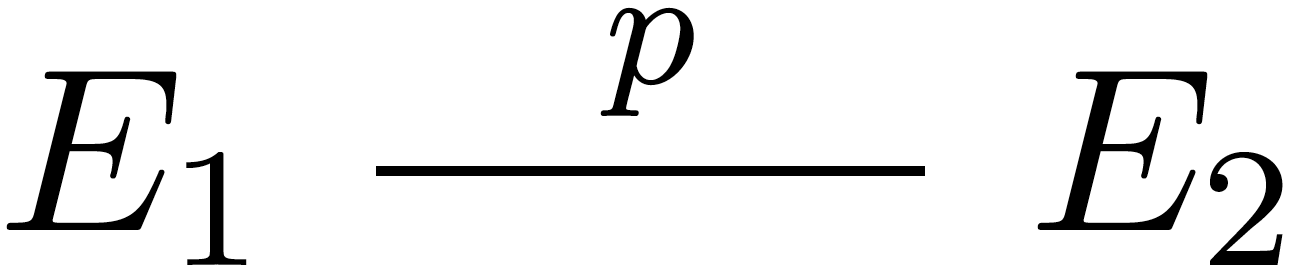}} & \multirow{12}*{$2$} & \multirow{12}*{$([2],[2])$} & \multirow{2}*{$\left\langle 3 \cdot \operatorname{Id}, \begin{bmatrix} 1 & 0 \\ 0 & -1 \end{bmatrix}, \begin{bmatrix} 1 & 1 \\ -2 & 1 \end{bmatrix} \right\rangle$} & \multirow{2}*{$\left\langle 3 \cdot \operatorname{Id}, \begin{bmatrix} -1 & 0 \\ 0 & 1 \end{bmatrix}, \begin{bmatrix} 1 & 1 \\ -2 & 1 \end{bmatrix} \right\rangle$} & \multirow{2}*{\texttt{256.a}} \\
    & & & & & \\
    \cline{4-6}
    & & & \multirow{2}*{$\mathcal{N}_{-2,0}(2^{\infty})$} & \multirow{2}*{$\mathcal{N}_{-2,0}(2^{\infty})$} & \multirow{2}*{\texttt{2304.h}} \\
    & & & & & \\
    \cline{4-6}
    & & & \multirow{2}*{$\left\langle -\operatorname{Id}, 3 \cdot \operatorname{Id}, \begin{bmatrix} 2 & 1 \\ -1 & 2 \end{bmatrix}, \begin{bmatrix} 1 & 0 \\ 0 & -1 \end{bmatrix} \right\rangle$} & \multirow{2}*{$\left\langle -\operatorname{Id}, 3 \cdot \operatorname{Id}, \begin{bmatrix} 2 & 1 \\ -1 & 2 \end{bmatrix}, \begin{bmatrix} 0 & 1 \\ 1 & 0 \end{bmatrix} \right\rangle$} & \multirow{2}*{\texttt{2304.a}} \\
    & & & & & \\
    \cline{4-6}
    & & & \multirow{2}*{$\left\langle 3 \cdot \operatorname{Id}, \begin{bmatrix} 2 & -1 \\ 1 & 2 \end{bmatrix}, \begin{bmatrix} 1 & 0 \\ 0 & -1 \end{bmatrix} \right\rangle$} & \multirow{2}*{$\left\langle 3 \cdot \operatorname{Id}, \begin{bmatrix} 2 & -1 \\ 1 & 2 \end{bmatrix}, \begin{bmatrix} 0 & 1 \\ 1 & 0 \end{bmatrix} \right\rangle$} & \multirow{2}*{\texttt{256.c}} \\
    & & & & & \\
    \cline{4-6}
    & & & \multirow{2}*{$\left\langle 3 \cdot \operatorname{Id}, \begin{bmatrix} -2 & 1 \\ -1 & -2 \end{bmatrix}, \begin{bmatrix} 1 & 0 \\ 0 & -1 \end{bmatrix} \right\rangle$} & \multirow{2}*{$\left\langle 3 \cdot \operatorname{Id}, \begin{bmatrix} -2 & 1 \\ -1 & -2 \end{bmatrix}, \begin{bmatrix} 0 & 1 \\ 1 & 0 \end{bmatrix} \right\rangle$} & \multirow{2}*{\texttt{256.b}} \\
    & & & & & \\
    \cline{4-6}
    & & & \multirow{2}*{$\mathcal{N}_{-1,0}(2^{\infty})$} & \multirow{2}*{$\mathcal{N}_{-1,0}(2^{\infty})$} & \multirow{2}*{\texttt{288.a}} \\
    & & & & & \\
    \cline{2-6}
    & \multirow{4}*{$3$} & \multirow{2}*{$([3],[1])$} & \multirow{4}*{$\mathcal{N}_{-1,1}(2^{\infty})$} & \multirow{4}*{$\mathcal{N}_{-1,1}(2^{\infty})$} & \multirow{2}*{\texttt{108.a}} \\
    & & & & & \\
    \cline{3-3}
    \cline{6-6}
    & & \multirow{2}*{$([1],[1])$} & & & \multirow{2}*{\texttt{225.c}} \\
    & & & & & \\
    \cline{2-6}
    & \multirow{2}*{$11$} & \multirow{10}*{$([1],[1])$} & \multirow{2}*{$\mathcal{N}_{-3,1}(2^{\infty})$} & \multirow{2}*{$\mathcal{N}_{-3,1}(2^{\infty})$} & \multirow{2}*{\texttt{121.b}} \\
    & & & & & \\
    \cline{2-2}
    \cline{4-6}
    & \multirow{2}*{$19$} & & \multirow{2}*{$\mathcal{N}_{-5,1}(2^{\infty})$} & \multirow{2}*{$\mathcal{N}_{-5,1}(2^{\infty})$} & \multirow{2}*{\texttt{361.a}} \\
    & & & & & \\
    \cline{2-2}
    \cline{4-6}
    & \multirow{2}*{$43$} & & \multirow{2}*{$\mathcal{N}_{-11,1}(2^{\infty})$} & \multirow{2}*{$\mathcal{N}_{-11,1}(2^{\infty})$} & \multirow{2}*{\texttt{1849.b}} \\
    & & & & & \\
    \cline{2-2}
    \cline{4-6}
    & \multirow{2}*{$67$} & & \multirow{2}*{$\mathcal{N}_{-17,1}(2^{\infty})$} & \multirow{2}*{$\mathcal{N}_{-17,1}(2^{\infty})$} & \multirow{2}*{\texttt{4489.b}} \\
    & & & & & \\
    \cline{2-2}
    \cline{4-6}
    & \multirow{2}*{$163$} & & \multirow{2}*{$\mathcal{N}_{-41,1}(2^{\infty})$} & \multirow{2}*{$\mathcal{N}_{-41,1}(2^{\infty})$} & \multirow{2}*{\texttt{26569.a}} \\
    & & & & & \\
\hline
\end{tabular}}
	\caption{}
 	\label{TableB}
\end{table}
\end{center}
\end{proposition}

\begin{ack}
 	The author would like to express his gratitude to \'Alvaro Lozano-Robledo for many helpful conversations about this topic and providing invaluable code to help with computations.
 \end{ack}

\section{Background and some lemmas}\label{sec-background}

\subsection{Elliptic curves, isogeny graphs, and isogeny-torsion graphs}

Let $E/\QQ$ be an elliptic curve. Then $E$ has the structure of an abelian group. Let $N$ be a positive integer. The set of points on $E$ of order dividing $N$ with coordinates in $\overline{\QQ}$ is a group, denoted $E[N]$ and is isomorphic to $\ZZ / N \ZZ \times \ZZ / N \ZZ$. An element of $E[N]$ is called an $N$-torsion point. Let $E/\QQ$ and $E'/\QQ$ be elliptic curves. An isogeny mapping $E$ to $E'$ is a non-constant morphism $\phi \colon E \to E'$ that maps the identity of $E$ to the identity of $E'$. An isogeny is a group homomorphism with a kernel of finite order. The degree of an isogeny agrees with the order of its kernel.

Let $M$ be an integer and let $[M] \colon E \to E$ be the map such that
\begin{center}
    $\begin{cases}
    [M](P) = \underbrace{P+ \ldots +P}_{\text{M}} & M \geq 1 \\
    [M](P) = \underbrace{(-P) + \ldots +(-P)}_{-M} & M \leq -1 \\
    [M](P) = \mathcal{O} & M = 0
    \end{cases}$
\end{center}
We call the map $[M]$ the multiplication-by-$M$ map. The endomorphism ring of $E$ is the set of all isogenies mapping $E$ to $E$, denoted, $\operatorname{End}(E)$. All of the multiplication-by-$M$ maps are elements of $\operatorname{End}(E)$. If $\operatorname{End}(E)$ consists solely of the multiplication-by-$M$ maps, then $\operatorname{End}(E)$ is ring-isomorphic to $\ZZ$ and $E$ is said to not have complex multiplication (CM). Otherwise, $E$ has CM and $\operatorname{End}(E)$ is isomorphic as a ring to an order in a quadratic field.

\begin{example}
Let $E$ be the elliptic curve with LMFDB label \texttt{11.a1}. Then $E$ does not have CM. In other words, $\operatorname{End}(E) \cong \ZZ$.
\end{example}

\begin{example}
Let $E$ be the elliptic curve $y^{2} = x^{3} - x$. Consider the isogeny $[i] \colon E \to E$ that maps $\mathcal{O}$ to $\mathcal{O}$ and maps a point $(a,b)$ in $E$ to the point $(-a,ib)$. Thus, $[i]$ maps non-zero points on $E$ to non-zero points on $E$ and hence, the degree of $[i]$ is equal to $1$. As $[i]$ is not equal to the identity or inversion maps, $[i]$ is an endomorphism of $E$ that is not a multiplication-by-$M$ map. Hence, $E$ has CM and $\operatorname{End}(E) = \ZZ + [i] \cdot \ZZ \cong \ZZ[i]$. Note that $i$ in $\ZZ + [i] \cdot \ZZ$ designates the map $[i]$ and the $i$ in $\ZZ[i]$ designates a root of $x^{2}+1$. 
\end{example}

Let $E/\QQ$ be a homogenized elliptic curve. The group $G_{\QQ}:= \operatorname{Gal}(\overline{\QQ}/\QQ)$ has a natural action on $E[N]$ for all positive integers $N$; for each $[a,b,c] \in E$ and each $\sigma \in G_{\QQ}$, we have
$$\sigma \cdot [a,b,c] = [\sigma(a),\sigma(b),\sigma(c)].$$
From this action, we have the mod-$N$ Galois representation attached to $E$:
$$ \overline{\rho}_{E,N} \colon G_{\QQ} \to \operatorname{Aut}(E[N]).$$
After identifying $E[N] \cong \ZZ / N \ZZ \times \ZZ / N \ZZ$ and fixing a set of (two) generators of $E[N]$, we may consider the mod-$N$ Galois representation attached to $E$ as
$$\overline{\rho}_{E,N} \colon G_{\QQ} \to \operatorname{GL}(2,\ZZ / N \ZZ).$$
Let $\ell$ be a prime and denote $\rho_{E,\ell^{\infty}}(G_{\QQ}) = \varprojlim \overline{\rho}_{E,\ell^{N}}(G_{\QQ})$. In particular, the group $\rho_{E,2^{\infty}}(G_{\QQ})$ is the main focus of this paper. Let $u$ be an element of $\left(\ZZ / N \ZZ \right)^{\times}$. By the properties of the Weil pairing, there exists an element of $\overline{\rho}_{E,N}(G_{\QQ})$ whose determinant is equal to $u$. Moreover, $\overline{\rho}_{E,N}(G_{\QQ})$ has an element that behaves like complex conjugation.

\begin{definition}
Let $E/\QQ$ be a (homogenized) elliptic curve. A point $P$ on $E$ is said to be defined over $\QQ$ if $P = [a:b:c]$ for some $a,b,c \in \QQ$.
\end{definition}
The set of all points on $E$ defined over $\QQ$ is denoted $E(\QQ)$. By the Mordell--Weil theorem, $E(\QQ)$ has the structure of a finitely-generated abelian group. Let $E(\QQ)_{\text{tors}}$ denote the set of points on $E$ defined over $\QQ$ of finite order.

\begin{thm}[Mazur \cite{mazur1}]\label{thm-mazur}
		Let $E/\Q$ be an elliptic curve. Then
		\[
		E(\Q)_\tor\simeq
		\begin{cases}
		\Z/M\Z &\text{with}\ 1\leq M\leq 10\ \text{or}\ M=12,\ \text{or}\\
		\Z/2\Z \times \Z/2N\Z &\text{with}\ 1\leq N \leq 4.
		\end{cases}
		\]
	\end{thm}
Moreover, each of the fifteen torsion subgroups occur for infinitely many \textit{j}-invariants. We now move on to the possible isogenies with finite, cyclic kernel.

\begin{defn}
Let $E/\QQ$ be an elliptic curve. A subgroup $H$ of $E$ of finite order is said to be $\QQ$-rational if $\sigma(H) = H$ for all $\sigma \in G_{\QQ}$.
\end{defn}

\begin{remark}

Note that for an elliptic curve $E / \QQ$, a group generated by a point $P$ on $E$ defined over $\QQ$ of finite order is certainly a $\QQ$-rational group but in general, the elements of a $\QQ$-rational subgroup of $E$ need not be \textit{fixed} by $G_{\QQ}$. For example, $E[3]$ is a $\QQ$-rational subgroup of $E$ of order $9$ and $G_{\QQ}$ fixes one or three of the nine elements of $E[3]$ by Theorem \ref{thm-mazur}.
\end{remark}

\begin{lemma}[III.4.12, \cite{Silverman}]\label{Q-rational}
Let $E/\QQ$ be an elliptic curve. Then for each finite, cyclic, $\QQ$-rational subgroup $H$ of $E$, there is a unique elliptic curve defined over $\QQ$ up to isomorphism denoted $E / H$, and an isogeny $\phi_{H} \colon E \to E / H$ with kernel $H$.
\end{lemma}

\begin{remark}

Note that it is only the elliptic curve $E/H$ that is unique (up to isomorphism) but the isogeny $\phi_{H}$ is not necessarily unique. For any isogeny $\phi$, the isogeny $-\phi$ has the same domain, codomain, and kernel as $\phi$. Moreover, for any positive integer $N$, $\phi$ and $[N] \circ \phi$ have the same domain and the same codomain. This is why the bijection in Lemma \ref{Q-rational} is with \textit{cyclic}, $\QQ$-rational subgroups instead of with all $\QQ$-rational subgroups.
\end{remark}

The $\QQ$-rational points on the modular curves $\operatorname{X}_{0}(N)$ have been described completely in the literature, for all $N\geq 1$. One of the most important milestones in the classification was \cite{mazur1}, where Mazur dealt with the case when $N$ is prime. The complete classification of $\Q$-rational points on $\operatorname{X}_{0}(N)$, for any $N$, was completed due to work by Fricke, Kenku, Klein, Kubert, Ligozat, Mazur  and Ogg, among others (see the summary tables in \cite{lozano0}).
	
	\begin{thm}\label{thm-ratnoncusps} Let $N\geq 2$ be a number such that $\operatorname{X}_{0}(N)$ has a non-cuspidal $\QQ$-rational point. Then:
		\begin{enumerate}
			\item $N\leq 10$, or $N= 12,13, 16,18$ or $25$. In this case $\operatorname{X}_{0}(N)$ is a curve of genus $0$ and its $\Q$-rational points form an infinite $1$-parameter family, or
			\item $N=11,14,15,17,19,21$, or $27$. In this case $\operatorname{X}_{0}(N)$ is a curve of genus $1$, i.e.,~$\operatorname{X}_{0}(N)$ is an elliptic curve over $\Q$, but in all cases the Mordell-Weil group $\operatorname{X}_{0}(N)(\Q)$ is finite, or 
			
			\item $N=37,43,67$ or $163$. In this case $\operatorname{X}_{0}(N)$ is a curve of genus $\geq 2$ and (by Faltings' theorem) there are only finitely many $\Q$-rational points, which are known explicitly.
		\end{enumerate}
	\end{thm}

\begin{defn}
		Let $E/\Q$ be an elliptic curve. We define $C(E)$ as the number of finite, cyclic, $\QQ$-rational subgroups of $E$ (including the trivial subgroup), and we define $C_p(E)$ similarly to $C(E)$ but only counting cyclic, $\QQ$-rational subgroups of order a power of $p$ (like in the definition of $C(E)$, this includes the trivial subgroup), for each prime $p$. 
	\end{defn}

Notice that it follows from the definition that $C(E)=\prod_p C_p(E)$. 
	
	\begin{thm}[Kenku, \cite{kenku}]\label{thm-kenku} There are at most eight $\Q$-isomorphism classes of elliptic curves in each $\Q$-isogeny class. More concretely, let $E / \Q$ be an elliptic curve, then $C(E)=\prod_p C_p(E)\leq 8$. Moreover, each factor $C_p(E)$ is bounded as follows:
		\begin{center}
			\begin{tabular}{c|ccccccccccccc}
				$p$ & $2$ & $3$ & $5$ & $7$ & $11$ & $13$ & $17$ & $19$ & $37$ & $43$ & $67$ & $163$ & \text{else}\\
				\hline 
				$C_p\leq $ & $8$ & $4$ & $3$ & $2$ & $2$ & $2$ & $2$ & $2$ & $2$ & $2$ & $2$ & $2$ & $1$.
			\end{tabular}
		\end{center}
		Moreover:
		\begin{enumerate}
			\item If $C_{p}(E) = 2$ for a prime $p$ greater than $7$, then $C_{q}(E) = 1$ for all other primes $q$. 
			\item Suppose $C_{7}(E) = 2$, then $C(E) \leq 4$. Moreover, we have $C_{3}(E) = 2$, or $C_{2}(E) = 2$, or $C(E) = 2$.
			\item $C_{5}(E) \leq 3$ and if $C_{5}(E) = 3$, then $C(E) = 3$.
			\item If $C_{5}(E) = 2$, then $C(E) \leq 4$. Moreover, either $C_{3}(E) = 2$, or $C_{2}(E) = 2$, or $C(E) = 2$. 
			\item $C_{3}(E) \leq 4$ and if $C_{3}(E) = 4$, then $C(E) = 4$. 
			\item If $C_{3}(E) = 3,$ then $C(E) \leq 6$. Moreover, $C_{2}(E) = 2$ or $C(E) = 3$.
			\item If $C_{3}(E) = 2$, then $C_{2}(E) \leq 4$.
		\end{enumerate}
	\end{thm}

Instead of viewing each elliptic curve over $\QQ$ in a $\QQ$-isogeny class individually, we can view them all together. The best way to visualize the $\QQ$-isogeny class is to use its associated isogeny graph.
\begin{thm}\label{thm-mainisogenygraphs}
	There are $26$ isomorphism types of isogeny graphs that are associated to $\QQ$-isogeny classes of elliptic curves defined over $\QQ$. More precisely, there are $16$ types of (linear) $L_{k}$ graphs, $3$ types of (nonlinear two-primary torsion) $T_{k}$ graphs, $6$ types of (rectangular) $R_{k}$ graphs, and $1$ type of (special) $S$ graph. Moreover, there are $11$ isomorphism types of isogeny graphs that are associated to $\QQ$-isogeny classes of elliptic curves over $\QQ$ with complex multiplication, namely the types $L_{2}(p)$ for $p=2,3,11,19,43,67,163$, $L_{4}$, $T_{4}$, $R_{4}(6)$, and $R_{4}(14)$. Finally, the isogeny graphs of type $L_{4}$, $R_{4}(14)$, and $L_{2}(p)$ for $p \in \{19, 43, 67, 167\}$ occur exclusively for elliptic curves with CM.
	\end{thm}
The main theorem in \cite{gcal-r} was the classification of isogeny-torsion graphs that occur over $\QQ$.
\begin{thm}\label{thm-main2} There are $52$ isomorphism types of isogeny-torsion graphs that are associated to $\QQ$-isogeny classes of elliptic curves defined over $\QQ$. In particular, there are $23$ types of $L_{k}$ graphs, $13$ types of $T_{k}$ graphs, $12$ types of $R_{k}$ graphs, and $4$ types of $S$ graphs. Moreover, there are $16$ isomorphism types of isogeny-torsion graphs that are associated to $\QQ$-isogeny classes of elliptic curves over $\QQ$ with complex multiplication (and examples are given in Table \ref{tab-CMgraphs}).
\end{thm}

The $16$ isomorphism types of isogeny-torsion graphs that occur over $\QQ$ with CM are the two isogeny-torsion graphs of type $L_{4}$
\begin{center}
    \begin{tikzcd}
\mathbb{Z} / M \mathbb{Z} \arrow[r, no head, "3"] & \mathbb{Z} / M \mathbb{Z} \arrow[r, no head, "3"] & \mathbb{Z} / M \mathbb{Z} \arrow[r, no head, "3"] & \mathcal{O}
\end{tikzcd}
\end{center}
where $m = 1$ or $3$, the eight isogeny-torsion graphs of type $L_{2}$
\begin{center}
    \begin{tikzcd}
\mathbb{Z} / 2 \mathbb{Z} \arrow[r, no head, "2"] & \mathbb{Z} / 2 \mathbb{Z}
\end{tikzcd}, \begin{tikzcd}
\mathbb{Z} / M \mathbb{Z} \arrow[r, no head, "3"] & \mathcal{O}
\end{tikzcd}, \begin{tikzcd}
\mathcal{O} \arrow[r, no head, "p"] & \mathcal{O}
\end{tikzcd}
\end{center}
where $M = 3$ or $1$, and $p = 11$, $19$, $43$, $67$, or $163$, the three isogeny-torsion graphs of $R_{4}$ type
\begin{center}
    \begin{tikzcd}
\mathbb{Z} / 2 \mathbb{Z} \arrow[rr, no head, "2"] \arrow[d, no head, "7"'] &  & \mathbb{Z} / 2 \mathbb{Z} \arrow[d, no head, "7"] \\
\mathbb{Z} / 2 \mathbb{Z} \arrow[rr, no head, "2"']                &  & \mathbb{Z} / 2 \mathbb{Z}           
\end{tikzcd}, \begin{tikzcd}
\mathbb{Z} / M \mathbb{Z} \arrow[rr, no head, "2"] \arrow[d, no head, "3"'] &  & \mathbb{Z} / M \mathbb{Z} \arrow[d, no head, "3"] \\
\mathbb{Z} / 2 \mathbb{Z} \arrow[rr, no head, "2"']                &  & \mathbb{Z} / 2 \mathbb{Z}           
\end{tikzcd}
\end{center}
where $M = 2$ or $6$, and the three isogeny-torsion graphs of $T_{4}$ type.
\begin{center}
    \begin{tikzcd}
                          & \mathbb{Z} / 2 \mathbb{Z}                                                                                  &                           \\
                          & \mathbb{Z} / 2 \mathbb{Z} \times \mathbb{Z} / 2 \mathbb{Z} \arrow[u, no head, "2"] \arrow[ld, no head, "2"'] \arrow[rd, no head, "2"] &                           \\
\mathbb{Z} / 2 \mathbb{Z} &                                                                                                            & \mathbb{Z} / 2 \mathbb{Z}
\end{tikzcd} \begin{tikzcd}
                          & \mathbb{Z} / 2 \mathbb{Z}                                                                                  &                           \\
                          & \mathbb{Z} / 2 \mathbb{Z} \times \mathbb{Z} / 2 \mathbb{Z} \arrow[u, no head, "2"] \arrow[ld, no head, "2"'] \arrow[rd, no head, "2"] &                           \\
\mathbb{Z} / 4 \mathbb{Z} &                                                                                                            & \mathbb{Z} / 2 \mathbb{Z}
\end{tikzcd} \begin{tikzcd}
                          & \mathbb{Z} / 2 \mathbb{Z}                                                                                  &                           \\
                          & \mathbb{Z} / 2 \mathbb{Z} \times \mathbb{Z} / 2 \mathbb{Z} \arrow[u, no head, "2"] \arrow[ld, no head, "2"'] \arrow[rd, no head, "2"] &                           \\
\mathbb{Z} / 4 \mathbb{Z} &                                                                                                            & \mathbb{Z} / 4 \mathbb{Z}
\end{tikzcd}
\end{center}
We continue with some lemmas and more background that will be used to classify the $2$-adic Galois images attached to isogeny-torsion graphs with CM.

\subsection{Quadratic Twists}

\begin{lemma}\label{group quadratic twists}
Let $N$ be a positive integer such that $\operatorname{GL}(2,\ZZ / N \ZZ)$ has a subgroup $H$ that does not contain $\operatorname{-Id}$. Let $H' = \left\langle \operatorname{-Id}, H \right\rangle$. Then $H' \cong \left\langle \operatorname{-Id} \right\rangle \times H$.
\end{lemma}

\begin{proof}
Note that $H$ is a subgroup of $H'$ of index $2$. Hence, the order of $H'$ and the order of $\left\langle \operatorname{-Id}, H \right\rangle$ are the same. Define
$$\psi \colon \left\langle \operatorname{-Id} \right\rangle \times H \to H'$$
by $\psi(x,h) = xh$. Then $\psi$ is a group homomorphism as $\operatorname{-Id}$ is in the center of $\operatorname{GL}(2,\ZZ/N\ZZ)$. We are done if we prove that $\psi$ is injective. Let $(x,h) \in \left\langle \operatorname{-Id} \right\rangle \times H$ such that $\psi(x,h) = xh = \operatorname{Id}$. Then $h = x^{-1} = x$ as the order of $x$ is equal to $1$ or $2$. As $\operatorname{-Id} \notin H$, $h$ cannot be $\operatorname{-Id}$. Hence, $h = x = \operatorname{Id}$ and so, $\psi$ is injective. \end{proof}

\begin{definition}
Let $G$ and $H$ be subgroups of $\operatorname{GL}(2, \ZZ_{2})$.Then we will say that $G$ and $H$ are quadratic twists if $G$ is the same as $H$, up to multiplication of some elements of $H$ by $\operatorname{-Id}$. Note that if $\left\langle G, \operatorname{-Id} \right\rangle =  \left\langle H, \operatorname{-Id} \right\rangle$, then $H$ and $G$ are quadratic twists.
\end{definition}

\begin{lemma}\label{index of quadratic twists}
Let $N$ be a positive integer, let $H$ be a subgroup of $\operatorname{GL}(2,\Z/N\Z)$, and let $H' = \left\langle H, \operatorname{-Id} \right\rangle$. Let $\chi$ be a character of $H$ of degree two. Then $\chi(H) = H'$ or $\chi(H)$ is a subgroup of $H'$ of index $2$.
\end{lemma}

\begin{proof}

The character $\chi$ multiplies some elements of $H$ by $\operatorname{-Id}$. If $\operatorname{-Id} \in \chi(H)$, then we can multiply all of the elements of $H$ that $\chi$ multiplied by $\operatorname{-Id}$ by $\operatorname{-Id}$ again, and recoup all of the elements of $H$. Thus, $\chi(H)$ is a subgroup of $H'$ that contains both $\operatorname{-Id}$ and $H$ and so, $\chi(H) = H'$.

On the other hand, let us say that $\operatorname{-Id} \notin \chi(H)$. Let $\chi(H)' = \left\langle \chi(H), \operatorname{-Id} \right\rangle$. By the same argument from before, we can multiply all of the elements of $H$ that $\chi$ multiplied by $\operatorname{-Id}$ by $\operatorname{-Id}$ again, and recoup all of the elements of $H$ in $\chi(H)'$. In other words, $\chi(H)' = \left\langle \operatorname{-Id}, \chi(H) \right\rangle = \left\langle \operatorname{-Id}, H \right\rangle = H'$. By Lemma \ref{group quadratic twists}, $H' = \left\langle \chi(H), \operatorname{-Id} \right\rangle \cong \left\langle \operatorname{-Id} \right\rangle \times \chi(H)$ and $\chi(H)$ is a subgroup of $H'$ of index $2$.
\end{proof}

Let $E : y^{2} = x^{3} + Ax + B$ be an elliptic curve and let $d$ be a non-zero integer. Then the quadratic twist of $E$ by $d$ is the elliptic curve $E^{(d)} : y^{2} = x^{3} + d^{2}Ax + d^{3}B$. Equivalently, $E^{(d)}$ is isomorphic to the elliptic curve defined by $dy^{2} = x^{3} + Ax + B$. Moreover, $E$ is isomorphic to $E^{(d)}$ over $\QQ(\sqrt{d})$ by the map $\phi \colon E \to E^{(d)}$
defined by fixing $\mathcal{O}$ and mapping any non-zero point $(a,b)$ on $E$ to $\left(a,\frac{b}{\sqrt{d}}\right)$.

\begin{corollary}\label{subgroups of index 2 contain -Id}
Let $N$ be a positive integer such that $\operatorname{GL}(2, \Z/N \Z)$ contains a subgroup $H$ such that all subgroups of $H$ of index $2$ contain $\operatorname{-Id}$. Suppose that there exists an elliptic curve $E/\QQ$ such that $\overline{\rho}_{E,N}(G_{\QQ})$ is conjugate to $H$. Let $E^{\chi}$ be a quadratic twist of $E$. Then $\overline{\rho}_{E^{\chi},N}(G_{\Q})$ is conjugate to $H$.
\end{corollary}

\begin{proof}
The group $H$ contains $\operatorname{-Id}$. By Lemma \ref{index of quadratic twists}, $\overline{\rho}_{E^{\chi},N}(G_{\Q})$ is conjugate to $H$ or is conjugate to a subgroup of $H$ of index $2$. Moreover, $\overline{\rho}_{E^{\chi},N}(G_{\Q})$ is the same as $\overline{\rho}_{E,N}(G_{\Q}) = H$, up to multiplication of some elements of $H$ by $\operatorname{-Id}$. As all subgroups of $H$ of index $2$ contain $\operatorname{-Id}$, we can just multiply the elements of $\overline{\rho}_{E^{\chi},N}(G_{\Q})$ that $\chi$ multiplied by $\operatorname{-Id}$ again by $\operatorname{-Id}$ to recoup all elements of $H$. Hence, $\overline{\rho}_{E^{\chi},N}(G_{\QQ})$ is conjugate to $H$.
\end{proof}

\begin{remark}
Let $N$ be a positive integer such that $\operatorname{GL}(2, \Z/N\Z)$ contains a subgroup $H$ that does not contain $\operatorname{-Id}$. Suppose there is an elliptic curve $E/\QQ$ such that $\overline{\rho}_{E,N}(G_{\QQ})$ is conjugate to $H' = \left\langle H, \operatorname{-Id} \right\rangle$. Then there is a quadratic twist $\chi$ such that $\overline{\rho}_{E^{\chi},N}(G_{\QQ})$ is conjugate to $H$ (see Remark 1.1.3  and Section 10 in \cite{Rouse2021elladicIO}). Conversely, if $\overline{\rho}_{E,N}(G_{\QQ})$ is conjugate to $H$, then there is a non-zero integer $d$ and a quadratic twist $E^{(d)}$ of $E$, such that $\overline{\rho}_{E^{(d)},N}(G_{\QQ})$ is conjugate to $H'$. Note that $\QQ(E[N])$ does not contain $\QQ(\sqrt{d})$.
\end{remark}

\subsection{Galois representations}

\begin{lemma}\label{ell-adic Galois images}

Let $E$ and $E'$ be elliptic curves defined over $\QQ$ that are $\QQ$-isogenous by an isogeny $\phi$ whose kernel is finite, cyclic, and $\QQ$-rational. Let $\ell$ be a prime and let $r$ be the non-negative integer such that $\ell^{r}$ is the greatest power of $\ell$ that divides the order of $\operatorname{Ker}(\phi)$. Let $m$ be a non-negative integer. Then there is a basis $\{P_{\ell^{m+r}}, Q_{\ell^{m+r}}\}$ of $E[\ell^{m+r}]$ such that

\begin{enumerate}
\item if $\sigma$ is a Galois automorphism of $\QQ$, then, there are integers $A$, $B$, $C$, and $D$, where $\overline{\rho}_{E,\ell^{m+r}}(\sigma) = \begin{bmatrix} A & C \\ B & D \end{bmatrix}$ and $\ell^{r}$ divides $C$,
\item $\{\phi([\ell^{r}]P_{\ell^{m+r}}), \phi(Q_{\ell^{m+r}})\}$ is a basis of $E'[\ell^{m}]$,
\item $\overline{\rho}_{E',\ell^{m}}(\sigma) = \begin{bmatrix} A & \frac{C}{\ell^{r}} \\ \ell^{r} \cdot B & D \end{bmatrix}$.
\end{enumerate}

\end{lemma}

\begin{proof}
We break up the proof into steps

\begin{enumerate}
\item Let $Q_{\ell^{r}}$ be an element of $\operatorname{Ker}(\phi)$ of order $\ell^{r}$ and let $Q_{\ell^{m+r}}$ be a point on $E$ such that $[\ell^{m}]Q_{\ell^{m+r}} = Q_{\ell^{r}}$. Let $P_{\ell^{m+r}}$ be a point on $E$ such that $E[\ell^{m+r}] = \left\langle P_{\ell^{m+r}}, Q_{\ell^{m+r}} \right\rangle$. Let $\sigma$ be a Galois automorphism of $\QQ$. Then there are integers $A$ and $B$ such that $\sigma(P_{\ell^{m+r}}) = [A]P_{\ell^{m+r}} + [B]Q_{\ell^{m+r}}$ and there are integers $C$ and $D$ such that $\sigma(Q_{\ell^{m+r}}) = [C]P_{\ell^{m+r}} + [D]Q_{\ell^{m+r}}$. Then
$$\sigma(Q_{\ell^{r}}) = \sigma([\ell^{m}]Q_{\ell^{m+r}}) = [\ell^{m}]\sigma(Q_{\ell^{m+r}}) = [\ell^{m}]([C]P_{\ell^{m+r}}+[D]Q_{\ell^{m+r}}) = [\ell^{m}C]P_{\ell^{m+r}} + [D]Q_{\ell^{r}}.$$
As $Q_{\ell^{r}}$ generates a $\QQ$-rational group, $\sigma(Q_{\ell^{r}}) = [\ell^{m}C]P_{\ell^{m+r}} + [D]Q_{\ell^{r}} \in \left\langle Q_{\ell^{r}} \right\rangle \subseteq \left\langle Q_{\ell^{m+r}} \right\rangle$. Thus, $[\ell^{m} C]P_{\ell^{m+r}} \in \left\langle Q_{\ell^{m+r}} \right\rangle$. As $\left\langle P_{\ell^{m+r}} \right\rangle \bigcap \left\langle Q_{\ell^{m+r}} \right\rangle = \{\mathcal{O}\}$, we have that $[\ell^{m}C]P_{\ell^{m+r}} = \mathcal{O}$. Thus, $\ell^{m+r}$ divides $\ell^{m}C$ and hence, $\ell^{r}$ divides $C$.

\item We claim that $E'[\ell^{m}] = \left\langle \phi([\ell^{r}]P_{\ell^{m+r}}), \phi(Q_{\ell^{m+r}}) \right\rangle$. We claim that the order of $\phi([\ell^{r}]P_{\ell^{m+r}})$ and the order of $\phi(Q_{\ell^{m+r}})$ are both equal to $\ell^{m}$. Note that $[\ell^{m}]\phi([\ell^{r}]P_{\ell^{m+r}}) = \phi([\ell^{m+r}]P_{\ell^{m+r}}) = \mathcal{O}$. Next, $[\ell^{m}]\phi(Q_{\ell^{m+r}}) = \phi([\ell^{m}]Q_{\ell^{m+r}}) = \phi(Q_{\ell^{r}}) = \mathcal{O}$. If $m = 0$, then we can move on. If $m$ is positive, then $m-1$ is a non-negative integer and
$$[\ell^{m-1}] \cdot \phi([\ell^{r}]P_{\ell^{m+r}}) = \phi([\ell^{m+r-1}]P_{\ell^{m+r}}).$$
If we claim that $\phi([\ell^{m+r-1}]P_{\ell^{m+r}}) = \mathcal{O}$, then $[\ell^{m+r-1}]P_{\ell^{m+r}} \in \left\langle Q_{\ell^{r}} \right\rangle \subseteq \left\langle Q_{\ell^{m+r}} \right\rangle$. The point $[\ell^{m+r-1}]P_{\ell^{m+r}}$ generates the subgroup of $\left\langle P_{\ell^{m+r}} \right\rangle$ of order $\ell$ and so cannot be contained in $\left\langle Q_{\ell^{m+r}} \right\rangle$ and so we arrive at a contradiction. Next,
$$[\ell^{m-1}] \cdot \phi(Q_{\ell^{m+r}}) = \phi([\ell^{m-1}]Q_{\ell^{m+r}}).$$
If we claim that $\phi([\ell^{m-1}]Q_{\ell^{m+r}}) = \mathcal{O}$, then $[\ell^{m-1}]Q_{\ell^{m+r}} \in \left\langle Q_{\ell^{r}} \right\rangle$ but this is 
a contradiction as the order of $[\ell^{m-1}]Q_{\ell^{m+r}}$ is equal to $r+1$ and 
the order of $Q_{\ell^{r}}$ is equal to $\ell^{r}$.

Now we will prove that $\left\langle \phi([\ell^{r}]P_{\ell^{m+r}}) \right\rangle \bigcap \left\langle \phi(Q_{\ell^{m+r}}) \right\rangle = \left\{\mathcal{O} \right\}$. Now let us say that there are integers $\alpha$ and $\beta$ such that $[\alpha]\phi([\ell^{r}]P_{\ell^{m+r}}) = [\beta]\phi(Q_{\ell^{m+r}})$. Then $\phi([\alpha\ell^{r}]P_{\ell^{m+r}}) = \phi([\beta]Q_{\ell^{m+r}})$. Hence, $[\alpha \ell^{r}]P_{\ell^{m+r}} - [\beta]Q_{\ell^{m+r}} \in \left\langle Q_{\ell^{r}} \right\rangle \subseteq \left\langle Q_{\ell^{m+r}} \right\rangle$ and hence, $[\alpha \ell^{r}]P_{\ell^{m+r}} \in \left\langle Q_{\ell^{m+r}} \right\rangle$. Thus, $[\alpha \ell^{r}]P_{\ell^{m+r}} = \mathcal{O}$ and hence, $[\alpha]\phi([\ell^{m}]P_{\ell^{m+r}}) = \phi([\alpha \ell^{m}]P_{\ell^{m+r}}) = \mathcal{O}$. This means that $\left\langle \phi([\ell^{r}]P_{\ell^{m+r}}) \right\rangle \bigcap \left\langle \phi(Q_{\ell^{m+r}}) \right\rangle = \{\mathcal{O}\}$.

\item Next, we see that
$$\sigma(\phi([\ell^{r}]P_{\ell^{m+r}})) = \phi([\ell^{r}]\sigma(P_{\ell^{m+r}})) = \phi([\ell^{r}]([A]P_{\ell^{m+r}}+[B]Q_{\ell^{m+r}})) = [A]\phi([\ell^{r}]P_{\ell^{m+r}})+[\ell^{r} \cdot B]\phi(Q_{\ell^{m+r}}).$$
Finally, we see that
$$\sigma(\phi(Q_{\ell^{m+r}})) = \phi(\sigma(Q_{\ell^{m+r}})) = \phi([C]P_{\ell^{m+r}}+[D]Q_{\ell^{m+r}})$$
$$= \phi\left(\left[\frac{C}{\ell^{r}}\right][\ell^{r}]P_{\ell^{m+r}}+[D]Q_{\ell^{m+r}}\right) = \left[\frac{C}{\ell^{r}}\right]\phi([\ell^{r}]P_{\ell^{m+r}})+[D]\phi(Q_{\ell^{m+r}}).$$
\end{enumerate}
\end{proof}

\begin{remark}
Let $E$ and $E'$ be elliptic curves defined over $\QQ$. Let $\phi \colon E \to E'$ be a $\QQ$-isogeny with a finite, cyclic, $\QQ$-rational kernel. Let $\ell^{r}$ be a greatest power of $\ell$ that divides the order of $\operatorname{Ker}(\phi)$. Let $m$ be a non-negative integer. Given $\overline{\rho}_{E,\ell^{m+r}}(G_{\QQ})$, we may use Lemma \ref{ell-adic Galois images} to compute $\overline{\rho}_{E',\ell^{m}}(G_{\QQ})$. Therefore, $\rho_{E',\ell^{\infty}}(G_{\QQ})$ is determined by $\rho_{E,\ell^{\infty}}(G_{\QQ})$ (and vice versa).

\end{remark}

\iffalse

\begin{remark}

Let us say that $\rho_{E,\ell^{\infty}}(G_{\QQ})$ is a group of level $\ell^{n}$ for some non-negative integer $n$. Then for all non-negative integers $k$, $\overline{\rho}_{E,\ell^{n+k}}(G_{\QQ})$ is simply the full lift of $\overline{\rho}_{E,\ell^{n}}(G_{\QQ})$ to level $k$. In other words, if we are given that the level of $\rho_{E,\ell^{\infty}}(G_{\QQ})$ is equal to $n$ and if we are given $\overline{\rho}_{E,\ell^{n}}(G_{\QQ})$, then we can compute $\overline{\rho}_{E',\ell^{m}}(G_{\QQ})$ for all elliptic curves $E'$ that are $\QQ$-rational to $E$ and all non-negative integers $m$, using Lemma \ref{ell-adic Galois images}.

\end{remark}

\fi

\begin{corollary}\label{coprime isogeny-degree}

Let $E$ and $E'$ be elliptic curves defined over $\QQ$ and let $\ell$ be a prime number. Suppose that $E$ is $\QQ$-isogenous to $E'$ by an isogeny that is defined over $\QQ$ with a finite, cyclic kernel of degree not divisible by $\ell$. Then $\rho_{E,\ell^{\infty}}(G_{\QQ})$ is conjugate to $\rho_{E',\ell^{\infty}}(G_{\QQ})$.

\end{corollary}

\begin{proof}

Use Lemma \ref{ell-adic Galois images} with $r = 0$.

\end{proof}

\begin{corollary}\label{contains scalars}
 
Let $\ell$ be a prime and let $E$ and $E'$ be elliptic curves defined over $\QQ$. Suppose that $E$ is $\QQ$-isogenous to $E'$ by an isogeny $\phi$ with a finite, cyclic, $\QQ$-rational kernel. Let $\alpha$ be an integer that is not divisible by $\ell$. Then $\rho_{E,\ell^{\infty}}(G_{\QQ})$ contains $s_{\alpha} = \begin{bmatrix} \alpha & 0 \\ 0 & \alpha \end{bmatrix}$ if and only if $\rho_{E',\ell^{\infty}}(G_{\QQ})$ contains $s_{\alpha}$.
    
\end{corollary}

\begin{proof}

Suppose that $\rho_{E,\ell^{\infty}}(G_{\QQ})$ contains $s_{\alpha}$. Let $r$ be the non-negative integer such that $\ell^{r}$ is the greatest power of $\ell$ that divides $\operatorname{Ker}(\phi)$ and let $m$ be a non-negative integer. Then $\overline{\rho}_{E,\ell^{m+r}}(G_{\QQ})$ contains $s_{\alpha}$. As $s_{\alpha}$ is in the center of $\operatorname{GL}(2, \ZZ / \ell^{m+r} \ZZ)$, it does not matter what basis we use for $E[\ell^{m+r}]$. By Lemma \ref{ell-adic Galois images}, $s_{\alpha}$ is an element of $\overline{\rho}_{E',\ell^{m}}(G_{\QQ})$. The converse is proved simply by switching the roles of $E$ and $E'$ and using the dual of $\phi$.
\end{proof}

\begin{corollary}\label{contains -Id}
 
Let $\ell$ be a prime and let $E$ and $E'$ be elliptic curves defined over $\QQ$. Suppose that $E$ is $\QQ$-isogenous to $E'$ by an isogeny $\phi$ with a finite, cyclic, $\QQ$-rational kernel. Then $\rho_{E,\ell^{\infty}}(G_{\QQ})$ contains $\operatorname{-Id}$ if and only if $\rho_{E',\ell^{\infty}}(G_{\QQ})$ contains $\operatorname{-Id}$.
    
\end{corollary}

\begin{proof}

Use Corollary \ref{contains scalars} with $\alpha = -1$.

\end{proof}

\begin{lemma}[Generalized Hensel's Lemma]\label{Hensel}

Let $p$ be a prime and let $f(x)$ be a polynomial with integer coefficients. Suppose that $f(a) \equiv 0 (\mod p^{j})$, $p^{\tau} \mid \mid f'(a)$, and that $j \geq 2\tau + 1$. Then there is a unique $t$ (modulo $p$) such that $f(a+tp^{j-\tau}) \equiv 0 (\mod p^{j+1})$.

\end{lemma}

\section{Classification of $2$-adic Galois images attached to elliptic curves defined over $\QQ$ with CM}\label{sec-Alvaros work}

In \cite{al-rCMGRs}, Lozano-Robledo classified the image of $\ell$-adic Galois representations attached to elliptic curves defined over $\QQ$ with CM for all primes $\ell$. In this section, we will briefly go over the results from \cite{al-rCMGRs} that are important in this paper.

For the rest of the paper, let $K = \QQ(\sqrt{d})$ be a quadratic imaginary field, let $\mathcal{O}_{K}$ be the ring of integers of $K$ with discriminant $\Delta_{K}$. Then $\Delta_{K} = d$ if $d$ is congruent to $1$ modulo $4$ and $\Delta_{K} = 4d$ otherwise. Let $f$ be a positive integer and let $\mathcal{O}_{K,f}$ be the order of $K$ of conductor $f$.

\begin{theorem}[Theorem 1.1, \cite{al-rCMGRs}]\label{Theorem 1.1}
Let $E/\QQ$ be an elliptic curve with CM by $\mathcal{O}_{K,f}$, let $N$ be an even integer greater than or equal to $4$ and let $\overline{\rho}_{E,N} \colon G_{\QQ} \to \operatorname{GL}(2, \ZZ / N \ZZ)$.
\begin{itemize}
    \item If $\Delta_{K} \cdot f^{2} \equiv 0 \mod 4$, then set $\delta = \frac{\Delta_{K} \cdot f^{2}}{4}$ and $\phi = 0$.
    
    \item If $\Delta_{K} \cdot f^{2} \equiv 1 \mod 4$, then set $\delta = \frac{(\Delta_{K}-1)}{4} \cdot f^{2}$ and $\phi = f$.
\end{itemize}
Define the group $\mathcal{C}_{\delta,\phi}(N)$ to be the subgroup of $\operatorname{GL}(2, \ZZ / N \ZZ)$ consisting of all matrices of the form $\begin{bmatrix} a+b \phi & b \\ \delta b & a \end{bmatrix}$ and define $\mathcal{N}_{\delta, \phi}(N)$ to be the group $\mathcal{N}_{\delta, \phi}(N) = \left\langle \mathcal{C}_{\delta, \phi}(N), \begin{bmatrix} -1 & 0 \\ \phi & 1 \end{bmatrix} \right\rangle$. Then
\begin{enumerate}
    \item there is a $\ZZ / N \ZZ$-basis of $E[N]$ such that $\overline{\rho}_{E,N}(G_{\QQ})$ is contained in $\mathcal{N}_{\delta, \phi}(N)$
    \item and $\mathcal{C}_{\delta, \phi}(N) \cong \left(\mathcal{O}_{K,f} / N \mathcal{O}_{K,f} \right)^{\times}$ is a subgroup of $\mathcal{N}_{\delta, \phi}(N)$ of index $2$.
\end{enumerate}
\end{theorem}

\begin{theorem}[Theorem 1.2, \cite{al-rCMGRs}]\label{Theorem 1.2}
Let $E/\QQ$ be an elliptic curve with CM by $\mathcal{O}_{K,f}$.
\begin{itemize}
    \item If $\Delta_{K} \cdot f^{2} \equiv 0 \mod 4$, then set $\delta = \frac{\Delta_{K} \cdot f^{2}}{4}$ and $\phi = 0$.
    
    \item If $\Delta_{K} \cdot f^{2} \equiv 1 \mod 4$, then set $\delta = \frac{(\Delta_{K}-1)}{4} \cdot f^{2}$ and $\phi = f$.
\end{itemize}
Let $\rho_{E}$ be the Galois representation $\rho_{E} \colon \operatorname{Gal}(\overline{\QQ}/\QQ) \to \varprojlim \operatorname{Aut}(E[N]) \cong \operatorname{GL}(2, \widehat{\ZZ})$ and let $\mathcal{N}_{\delta,\phi} = \varprojlim \mathcal{N}_{\delta,\phi}(N)$. Then there is a compatible system of bases of $E[N]$ such that the image of $\rho_{E}$ is contained in $\mathcal{N}_{\delta,\phi}$, and the index of the image of $\rho_{E}$ in $\mathcal{N}_{\delta,\phi}$ is a divisor of the order of $\mathcal{O}_{K,f}^{\times}$. In particular, the index is a divisor of $4$ or $6$.
\end{theorem}

From now on, let $H_{f} = K(\textit{j}_{K,f})$

\begin{theorem}[Theorem 1.6, \cite{al-rCMGRs}]\label{Theorem 1.6}
Let $E/\QQ(\textit{j}_{K,f})$ be an elliptic curve with CM by an order $\mathcal{O}_{\delta,f}$ in an imaginary field $K$ with $\textit{j}_{E} \neq 0, 1728$. Then, for every $m \geq 1$, we have $\operatorname{Gal}(H_{f}(E[2^{m}])/H_{f}) \subseteq \left( \mathcal{O}_{K,f} / 2^{m} \mathcal{O}_{K,f}\right)^{\times}$. Suppose that $\operatorname{Gal}(H_{f}(E[2^{n}]) / H_{f}) \subsetneq \left(\mathcal{O}_{K,f} / 2^{n}\mathcal{O}_{K,f}\right)^{\times}$ for some positive integer $n$ and assume that $n$ is the smallest such integer. Then $n \leq 3$ and for all $m \geq 3$, we have
$$\operatorname{Gal}(H_{f}(E[2^{m}]) / H_{f}) \cong \left(\mathcal{O}_{K,f} / 2^{m}\mathcal{O}_{K,f}\right)^{\times} / \{\pm 1\}.$$
Further, there are two possibilities:
\begin{enumerate}
    \item If $n \leq 2$, then $\operatorname{Gal}(H_{f}(E[4])/H_{f}) \cong \left(\mathcal{O}_{K,f} / 4 \mathcal{O}_{K,f}\right)^{\times} / \{\pm 1 \}$ and:
    \begin{enumerate}
        \item $disc(\left(\mathcal{O}_{K,f}\right) = \Delta_{K} \cdot f^{2} \equiv 0 \mod 16$. In particular, we have either
        \begin{itemize}
            \item $\Delta_{K} \equiv 1 \mod 4$ and $f \equiv 0 \mod 4$, or
            \item $\Delta_{K} \equiv 0 \mod 4$ and $f \equiv 0 \mod 2$.
        \end{itemize}
        \item $\QQ(i) \subseteq H_{f}$.
        \item For each $m \geq 2$, there is a $\ZZ / 2^{m} \ZZ$-basis of $E[2^{m}]$ such that the image of the Galos representation $\rho_{E,2^{m}} \colon \operatorname{Gal}(\overline{H}_{f} / H_{f}) \to \operatorname{GL}(2, \ZZ / 2^{m} \ZZ)$ is one of the groups
        \begin{center}
            $J_{1} = \left\langle \begin{bmatrix} 5 & 0 \\ 0 & 5 \end{bmatrix}, \begin{bmatrix} 1 & 1 \\ \delta & 1 \end{bmatrix} \right\rangle$ or $J_{2} = \left\langle \begin{bmatrix} 5 & 0 \\ 0 & 5 \end{bmatrix}, \begin{bmatrix} -1 & -1 \\ -\delta & -1 \end{bmatrix} \right\rangle \subseteq \mathcal{C}_{\delta,0}(2^{m})$.
        \end{center}
        \end{enumerate}
        \item If $n = 3$, then $\operatorname{Gal}(H_{f}(E[4]) / H_{f}) \cong \left(\mathcal{O}_{K} / 4 \mathcal{O}_{K,f}\right)^{\times}$ and:
        \begin{enumerate}
            \item $\Delta_{K} \equiv 0 \mod 8$.
            \item For each $m \geq 3$, there is a $\ZZ / 2^{m} \ZZ$-basis of $E[2^{m}]$ such that the image of the Galois representation $\rho_{E,2^{m}} \colon \operatorname{Gal}(\overline{H}_{f} / H_{f}) \to \operatorname{GL}(2, \ZZ / 2^{m} \ZZ)$ is the group
            \begin{center}
            $J_{1} = \left\langle \begin{bmatrix} 3 & 0 \\ 0 & 3 \end{bmatrix}, \begin{bmatrix} 1 & 1 \\ \delta & 1 \end{bmatrix} \right\rangle$ or $J_{2} = \left\langle \begin{bmatrix} 3 & 0 \\ 0 & 3 \end{bmatrix}, \begin{bmatrix} -1 & -1 \\ -\delta & -1 \end{bmatrix} \right\rangle \subseteq \mathcal{C}_{\delta,0}(2^{m})$.
        \end{center}
        \end{enumerate}
        \end{enumerate}
        Finally, there is some $\epsilon \in \{\pm 1 \}$ and $\alpha \in \{3, 5\}$ such that the image of $\rho_{E,2^{\infty}}$ is a conjugate of
        \begin{center}
            $\left\langle \begin{bmatrix} \epsilon & 0 \\ 0 & -\epsilon \end{bmatrix}, \begin{bmatrix} \alpha & 0 \\ 0 & \alpha \end{bmatrix}, \begin{bmatrix} 1 & \delta \\ 1 & 1 \end{bmatrix} \right\rangle$ or $\left\langle \begin{bmatrix} \epsilon & 0 \\ 0 & -\epsilon \end{bmatrix}, \begin{bmatrix} \alpha & 0 \\ 0 & \alpha \end{bmatrix}, \begin{bmatrix} -1 & -\delta \\ -1 & -1 \end{bmatrix} \right\rangle \subseteq \operatorname{GL}(2, \ZZ_{2})$.
        \end{center}
\end{theorem}

\begin{corollary}\label{Theorem 1.6 corollary 1}
Let $E/\QQ$ be an elliptic curve with CM by an order $\mathcal{O}_{K,f}$ in the number field $K$ with discriminant $\Delta_{K}$ and conductor $f$ and $\textit{j}_{E} \neq 0, 1728$. Let $m$ be a non-negative integer. Then $\overline{\rho}_{E,2^{m+3}}(G_{\QQ})$ is conjugate to the full lift of $\overline{\rho}_{E,8}(G_{\QQ})$ inside the group $\mathcal{N}_{\delta,\phi}\left(2^{3+m}\right)$.
\end{corollary}

\begin{corollary}\label{Theorem 1.6 corollary 2}
Let $E/\QQ$ be an elliptic curve with CM by an order $\mathcal{O}_{K,f}$ in the number field $K$ with discriminant $\Delta_{K}$ and conductor $f$ and $\textit{j}_{E} \neq 0, 1728$. If $\Delta_{K} \cdot f^{2}$ is not divisible by $8$, then $\rho_{E,2^{\infty}}(G_{\QQ})$ is conjugate to $\mathcal{N}_{\delta, \phi}(2^{\infty})$.
\end{corollary}

\begin{theorem}[Theorem 1.7, \cite{al-rCMGRs}]\label{Theorem 1.7}

Let $E/\QQ$ be an elliptic curve with $\textit{j}_{E} = 1728$ and let $c \in G_{\QQ}$ represent complex conjugation and $\gamma  = \rho_{E,2^{\infty}}(G_{\QQ})(c)$. Let $G_{E,2^{\infty}}$ be the image of $\rho_{E,2^{\infty}}$ and let $G_{E,K,2^{\infty}} = \rho_{E,2^{\infty}}(G_{\QQ(i)})$. Then, there is a $\ZZ_{2}$-basis of $T_{2}(E) = \varprojlim E[2^{n}]$ such that $G_{E,K,2^{\infty}}$ is one of the following groups:
\begin{itemize}
    \item If $[\mathcal{C}_{-1,0}(2^{\infty}) : G_{E,K,2^{\infty}}] = 1$, then $G_{E,K,2^{\infty}}$ is all of $\mathcal{C}_{-1,0}(2^{\infty})$, i.e.,
    $$G_{1} = \left\{\begin{bmatrix} a & b \\ -b & a \end{bmatrix} \in \operatorname{GL}(2, \ZZ_{2}) : a^{2}+b^{2} \not \equiv 0 \mod 2 \right\}.$$
    \item If $[\mathcal{C}_{-1,0}(2^{\infty}) : G_{E,K,2^{\infty}}] = 2$, then $G_{E,K,2^{\infty}}$ is one of the following groups:
    \begin{center}
        $G_{2,a} = \left\langle \operatorname{-Id}, 3 \cdot \operatorname{Id}, \begin{bmatrix} 1 & 2 \\ -2 & 1 \end{bmatrix} \right\rangle$ or $G_{2,b} = \left\langle \operatorname{-Id}, 3 \cdot \operatorname{Id}, \begin{bmatrix} 2 & 1 \\ -1 & 2 \end{bmatrix} \right\rangle$.
    \end{center}
    \item If $[\mathcal{C}_{-1,0}(2^{\infty}) : G_{E,K,2^{\infty}}] = 4$, then $G_{E,K,2^{\infty}}$ is one of the following groups
    \begin{center} $G_{4,a} = \left\langle 5 \cdot \operatorname{Id}, \begin{bmatrix} 1 & 2 \\ -2 & 1 \end{bmatrix} \right\rangle$, or $G_{4,b} = \left\langle 5 \cdot \operatorname{Id}, \begin{bmatrix} -1 & -2 \\ 2 & -1 \end{bmatrix} \right\rangle$ or \end{center}
    \begin{center} $G_{4,c} = \left\langle -3 \cdot \operatorname{Id}, \begin{bmatrix} 2 & -1 \\ 1 & 2 \end{bmatrix} \right\rangle$, or $G_{4,d} = \left\langle -3 \cdot \operatorname{Id}, \begin{bmatrix} -2 & 1 \\ -1 & -2 \end{bmatrix} \right\rangle$. \end{center}
\end{itemize}
Moreover, $G_{E,2^{\infty}} = \left\langle \gamma, G_{E,K,2^{\infty}} \right\rangle = \left\langle \gamma', G_{E,K,2^{\infty}} \right\rangle$ is generated by one of the groups above, and an element
$$\gamma' \in \left\{ c_{1} = \begin{bmatrix} 1 & 0 \\ 0 & -1 \end{bmatrix}, c_{-1} = \begin{bmatrix} -1 & 0 \\ 0 & 1 \end{bmatrix}, c_{1}' = \begin{bmatrix} 0 & 1 \\ 1 & 0 \end{bmatrix}, c_{-1}' = \begin{bmatrix} 0 & -1 \\ -1 & 0 \end{bmatrix} \right\},$$
such that $\gamma \equiv \gamma' \mod 4$.
\end{theorem}

\begin{theorem}[Theorem 1.8, \cite{al-rCMGRs}]\label{Theorem 1.8}

Let $E/\QQ$ be an elliptic curve with $\textit{j}_{E} = 0$, and let $c \in G_{\QQ}$ represent complex conjugation. Let $G_{E,2^{\infty}}$ be the image of $\rho_{E,2^{\infty}}$ and let $G_{E,K,2^{\infty}} = \rho_{E,2^{\infty}}(G_{\QQ\left(\sqrt{-3}\right)})$. Then there is a $\ZZ_{2}$-basis of $T_{2}(E)$ such that the image $G_{E,2^{\infty}}$ of $\rho_{E,2^{\infty}}$ is one of the following groups of $\operatorname{GL}(2, \ZZ_{2})$, with $\gamma = \rho_{E,2^{\infty}}(c)$.
\begin{itemize}
    \item Either, $[\mathcal{C}_{-1,1}(2^{\infty}) : G_{E,K,2^{\infty}}] = 3$, and
    $$G_{E,2^{\infty}} = \left\langle \gamma', \operatorname{-Id}, \begin{bmatrix} 7 & 4 \\ -4 & 3 \end{bmatrix}, \begin{bmatrix} 3 & 6 \\ -6 & -3 \end{bmatrix} \right\rangle$$
    $$= \left\langle \gamma', \left\{ \begin{bmatrix} a + b & b \\ -b & a \end{bmatrix} \in \operatorname{GL}(2, \ZZ_{2}) : a \not \equiv 0 \mod 2, b \equiv 0 \mod 2 \right\} \right\rangle,$$
    and $\left\{ \begin{bmatrix} a + b & b \\ -b & a \end{bmatrix} \in \operatorname{GL}(2, \ZZ_{2}) : a \not \equiv 0 \mod 2, b \equiv 0 \mod 2 \right\}$ is precisely the set of matrices that correspond to the subgroup of cubes of Cartan elements $\mathcal{C}_{-1,1}(2^{\infty})^{3}$ which is the unique group of index $3$ in $\mathcal{C}_{-1,1}(2^{\infty})$.
    \item Or, $[\mathcal{C}_{-1,1}(2^{\infty}) : G_{E,K,2^{\infty}}] = 1$, and
    $$G_{E,2^{\infty}} = \mathcal{N}_{-1,1}(2^{\infty}) = \left\langle \gamma', \operatorname{-Id}, \begin{bmatrix} 7 & 4 \\ -4 & 3 \end{bmatrix}, \begin{bmatrix} 2 & 1 \\ -1 & 1 \end{bmatrix} \right\rangle$$
    $$= \left\langle \gamma', \left\{\begin{bmatrix} a + b & b \\ -b & a \end{bmatrix} \in \operatorname{GL}(2, \ZZ_{2}) : a \not \equiv 0 \mod 2 \texttt{or} b \not \equiv 0 \mod 2 \right\} \right\rangle$$
    where $\gamma' \in \left\{ \begin{bmatrix} 0 & 1 \\ 1 & 0 \end{bmatrix}, \begin{bmatrix} 0 & -1 \\ -1 & 0 \end{bmatrix} \right\}$, and $\gamma \equiv \gamma' \mod 4$.
\end{itemize}

\end{theorem}

\begin{corollary}\label{points of order 2 with j=0}
Let $E/\QQ$ be an elliptic curve with $\textit{j}_{E} = 0$. Then $E$ has a point of order $2$ defined over $\QQ$ if and only if $\rho_{E,2^{\infty}}(G_{\QQ})$ is conjugate to $\left\langle \operatorname{-Id}, \begin{bmatrix} 0 & 1 \\ 1 & 0 \end{bmatrix}, \begin{bmatrix} 7 & 4 \\ -4 & 3 \end{bmatrix}, \begin{bmatrix} 3 & 6 \\ -6 & -3 \end{bmatrix} \right\rangle$ and $E$ does not have a point of order $2$ defined over $\QQ$ if and only if $\rho_{E,2^{\infty}}(G_{\QQ})$ is conjugate to $\mathcal{N}_{-1,1}(2^{\infty})$.
\end{corollary}

\begin{proof}
From the fact that $\textit{j}_{E} = 0$, $\rho_{E,2^{\infty}}(G_{\QQ})$ is conjugate to one of two groups. The reduction of the group $\left\langle \operatorname{-Id}, \begin{bmatrix} 0 & 1 \\ 1 & 0 \end{bmatrix}, \begin{bmatrix} 7 & 4 \\ -4 & 3 \end{bmatrix}, \begin{bmatrix} 3 & 6 \\ -6 & -3 \end{bmatrix} \right\rangle$ modulo $2$ is a group of order $2$ and the reduction of $\mathcal{N}_{-1,1}(2^{\infty})$ modulo $2$ is a group of order $6$. Hence, if $\rho_{E,2^{\infty}}(G_{\QQ})$ is conjugate to the former, $E$ has a point of order $2$ defined over $\QQ$ and if $\rho_{E,2^{\infty}}(G_{\QQ})$ is conjugate to the latter, then $E$ does not have a point of order $2$ defined over $\QQ$.
\end{proof}

\section{$2$-adic Galois images attached to isogeny-torsion graphs with CM}\label{proofs}
Here we classify the $2$-adic Galois image attached to isogeny-torsion graphs defined over $\QQ$ with CM. We will categorize the proofs based first on isogeny-torsion graphs and then on \textit{j}-invariant.

\begin{proposition}
Let $E/\QQ$ be an elliptic curve with complex multiplication such that the isogeny graph associated to the $\QQ$-isogeny class of $E$ is of $L_{4}$ type or of $L_{2}(3)$ type. Then $\rho_{E,2^{\infty}}(G_{\QQ})$ is conjugate to $\mathcal{N}_{-1,1}(2^{\infty})$.
\end{proposition}

\begin{proof}
If the isogeny graph associated to the $\QQ$-isogeny class of $E$ is of $L_{4}$ type, then it looks like the one below with the \textit{j}-invariant of the corresponding elliptic curves listed:
\begin{center}
\begin{tikzcd}
{E_{1}, \textit{j}_{E_{1}} = -12288000} \arrow[rr, no head, "3"] &  & {E_{2}, \textit{j}_{E_{2} = 0}} \arrow[rr, no head, "3"] &  & {E_{3}, \textit{j}_{E_{3} = 0}} \arrow[rr, no head, "3"] &  & {E_{4}, \textit{j}_{E_{4}} = -12288000}
\end{tikzcd}
\end{center}
and if the isogeny graph associated to the $\QQ$-isogeny class of $E$ is of $L_{2}(3)$ type, then it looks like the one below with the \textit{j}-invariant of the corresponding elliptic curves listed:

\begin{center}
    \begin{tikzcd}
{E, \textit{j}_{E_{1} = 0}} \arrow[r, no head, "3"] & {E_{2}, \textit{j}_{E_{2} = 0}}
\end{tikzcd}
\end{center}
Let $E'/\QQ$ be an elliptic curve that is $3$-isogenous to $E$ with $\textit{j}_{E'} = 0$. By Corollary \ref{coprime isogeny-degree}, $\rho_{E,2^{\infty}}(G_{\QQ})$ is conjugate to $\rho_{E',2^{\infty}}(G_{\QQ})$. As $E'$ does not have a point of order $2$ defined over $\QQ$, Corollary \ref{points of order 2 with j=0} shows that $\rho_{E,2^{\infty}}(G_{\QQ})$ is conjugate to $\mathcal{N}_{-1,1}(2^{\infty})$. 
\end{proof}

\begin{proposition}
Let $E/\QQ$ be an elliptic curve with CM by a number field $K$ with discriminant $\Delta_{K}$. Suppose that the isogeny graph associated to the $\QQ$-isogeny class of $E$ is of $L_{2}(p)$ type with $p \in \{11, 19, 43, 67, 163\}$. Then $\rho_{E,2^{\infty}}(G_{\Q})$ is conjugate to $\mathcal{N}_{\frac{\Delta_{K}-1}{4},1}(2^{\infty})$.
\end{proposition}

\begin{proof}

The elliptic curve $E$ is $p$-isogenous to an elliptic curve $E'/\QQ$. Moreover, $\textit{j}_{E} = \textit{j}_{E'}$ and $\textit{j}_{E} \neq 0, 1728$, meaning that $E$ is a quadratic twist of $E'$. The isogeny graph associated to the $\QQ$-isogeny class of $E$ is of $L_{2}(p)$ type (see below).

\begin{center}
    \begin{tikzcd}
E \arrow[r, "p", no head] & E'
\end{tikzcd}
\end{center}

By Corollary \ref{coprime isogeny-degree}, $\rho_{E,2^{\infty}}(G_{\QQ})$ is conjugate to $\rho_{E',2^{\infty}}(G_{\QQ})$. We prove that $\rho_{E,2^{\infty}}(G_{\QQ})$ is unaffected by quadratic twisting by showing that $\operatorname{-Id}$ is an element of every subgroup of $\rho_{E,2^{\infty}}(G_{\QQ})$ of index $2$.

By Table \ref{tab-CMgraphs}, $\textit{j}_{E} \in \left\{-32768, -884736, -884736000, -147197952000, -262537412640768000 \right\}$. Each such elliptic curve has CM by a quadratic imaginary field $K$ of discriminant $\Delta_{K}$ ($ = -11$, $-19$, $-43$, $-67$, and $-163$, respectively). We take an example of each such elliptic curve $E'/\QQ$ with \textit{j}-invariant equal to one of the five above, namely, the elliptic curves with LMFDB labels \texttt{121.b1}, \texttt{361.a1}, \texttt{1849.b1}, \texttt{4489.b1}, and \texttt{26569.a1}, respectively. By the fact that $\textit{j}_{E} = \textit{j}_{E'}$ and $\textit{j}_{E} \neq 0, 1728$, $E$ is a quadratic twist of $E'$. Running code provided by Lozano-Robledo, we see that the conductor of each of the elliptic curves $E'$ is equal to $f = 1$. Hence, $\Delta_{K} \cdot f^{2} \equiv 1 \mod 4$ and so, $\delta = \frac{\Delta_{K}-1}{4}$ and $\phi = 1$.

By the fact that $\Delta_{K} \cdot f^{2}$ is not divisible by $8$, Corollary \ref{Theorem 1.6 corollary 2} shows that, $\rho_{E',2^{\infty}}(G_{\QQ})$ is conjugate to $\mathcal{N}_{\delta,1}(2^{\infty}) = \left\langle C_{\delta,1}(2^{\infty}), \begin{bmatrix} -1 & 0 \\ 1 & 1 \end{bmatrix} \right\rangle$ where $C_{\delta,1}(2^{\infty}) = \left\{ \begin{bmatrix} a+b & b \\ \delta b & a \end{bmatrix} \colon a,b \in \ZZ_{2} | \texttt{a and b not both even} \right\}$. Note that setting $a = -1$ and $b = 0$ shows that $\operatorname{-Id} \in \rho_{E',2^{\infty}}(G_{\Q})$ for each of the five elliptic curves. Let $H$ be a subgroup of $\rho_{E',2^{\infty}}(G_{\Q})$ of index $2$. Then $H$ is normal and hence, the squares of all elements of $\rho_{E',2^{\infty}}(G_{\Q})$ are contained in $H$. Let $a$ be an integer and let $b = 1$. Then
$$\left(\begin{bmatrix} -1 & 0 \\ 1 & 1 \end{bmatrix} \cdot \begin{bmatrix} a+1 & 1 \\ \delta & a \end{bmatrix}\right)^{2} = \begin{bmatrix} a^{2}+a-\delta & 0 \\ 0 & a^{2}+a - \delta \end{bmatrix}.$$
We have to show that we have an equality of the form $-1 = a^{2}+a-\delta$ modulo $2^{N}$ for all non-negative integers $N$. Let $p(x) = x^{2} + x -  \delta + 1$. Then $p(0) = 0$ in $\ZZ / 2 \ZZ$ because $\delta$ is odd and $p'(0) = 1 \neq 0 \mod 2$. By Hensel's lemma, there is a unique solution $\alpha \in \ZZ_{2}$ to the equality $x^{2} + x - \delta = -1$. Hence, $\operatorname{-Id}$ is a square in $\rho_{E',2^{\infty}}(G_{\Q})$ and so, $H$ contains $\operatorname{-Id}$. By Corollary \ref{subgroups of index 2 contain -Id} and Corollary \ref{coprime isogeny-degree}, $\rho_{E,2^{\infty}}(G_{\QQ})$ is conjugate to $\rho_{E',2^{\infty}}(G_{\QQ})$. Hence, $\rho_{E,2^{\infty}}(G_{\QQ})$ is conjugate to $\mathcal{N}_{\delta,1}(2^{\infty}) = \mathcal{N}_{\frac{\Delta_{K}-1}{4},1}(2^{\infty})$.
\end{proof}

\begin{proposition}
Let $E/\QQ$ be an elliptic curve such that the isogeny graph associated to the $\QQ$-isogeny class of $E$ is of $R_{4}(14)$ type. Then $\textit{j}_{E} = 16581375$ or $\textit{j}_{E} = -3375$. In the former case, $\rho_{E,2^{\infty}}(G_{\QQ})$ is conjugate to $\mathcal{N}_{-7,0}(2^{\infty})$ and in the latter case, $\rho_{E,2^{\infty}}(G_{\QQ})$ is conjugate to $\mathcal{N}_{-2,1}(2^{\infty})$.
\end{proposition}

\begin{proof}
Let $E/\QQ$ be an elliptic curve that has a cyclic, $\QQ$-rational subgroup of order $14$. Then the isogeny graph associated to the $\QQ$-isogeny class of $E$ is below:
\begin{center}
\begin{tikzcd}
{E_{1}, \textit{j}_{E_{1}} = 16581375} \arrow[dd, no head, "7"'] \arrow[rr, no head, "2"] &  & {E_{2}, \textit{j}_{E_{2}} = -3375} \arrow[dd, no head, "7"] \\
                                                                     &  &                                                        \\
{E_{3}, \textit{j}_{E_{3}} = 16581375} \arrow[rr, no head, "2"']                 &  & {E_{4}, \textit{j}_{E_{4}} = -3375}               
\end{tikzcd}
\end{center}
Let $E/\QQ$ be the elliptic curve with LMFDB notation \texttt{49.a1}. Then $\textit{j}_{E} = 16581375$. Running code provided by Lozano-Robledo, we see that $E$ has complex multiplication by an order of $K = \QQ(\sqrt{-7})$ with discriminant $\Delta_{K} = -7$ and conductor $f = 2$. Then $\delta = \frac{-7 \cdot 2^{2}}{4} = -7$. By Corollary \ref{Theorem 1.6 corollary 2}, $\rho_{E,2^{\infty}}(G_{\Q})$ is conjugate to $\mathcal{N}_{-7,0}(2^{\infty}) = \left\langle C_{-7,0}(2^{\infty}), \begin{bmatrix} -1 & 0 \\ 0 & 1 \end{bmatrix} \right\rangle$ where $C_{-7,0}(2^{\infty}) = \left\langle \begin{bmatrix} a & b \\ -7b & a \end{bmatrix} \colon 2 \nmid a^{2}+b^{2} \right\rangle$.
We will prove that $\operatorname{-Id}$ is contained in all subgroups of $\mathcal{N}_{-7,0}(2^{\infty})$ of index $2$. Let $a = 0$ and let $b$ be an integer. Then
$$\left(\begin{bmatrix} -1 & 0 \\ 0 & 1 \end{bmatrix} \cdot \begin{bmatrix} 0 & b \\ -7b & 0 \end{bmatrix}\right)^{2} = \begin{bmatrix} 7b^{2} & 0 \\ 0 & 7b^{2} \end{bmatrix}.$$
We have to prove that for each positive integer $N$, there is a solution to the equation $-1 = 7b^{2}$.

Let $p = 2$ and consider the polynomial $f(x) = 7x^{2}+1$. We will use Lemma \ref{Hensel}. Let $a = 1$. Then $f(a) = 8 \equiv 0 \mod 2^{3}$. Next we have $f'(x) = 14x$ and $f'(a) = 14$. Letting $\tau = 1$, we have $2^{1} \lvert \rvert f'(a)$. Setting $j = 3$, we have that $j \geq 2\tau + 1$. So there is a unique integer $t$ modulo $2$ such that $f(1+t \cdot 2^{2}) \equiv 0 \mod 2^{4}$.

Now let $a_{1} = 1 + t \cdot 2^{2}$ and let $j = 4$. Then $f(a_{1}) \equiv 0 \mod 2^{4}$. Next, we have that because $a_{1}$ is odd and $f'(x) = 14x$, that $2^{1} \lvert \rvert f'(a_{1})$. Thus, $j \geq 2 \cdot \tau + 1$ and by Lemma \ref{Hensel}, there is a unique integer $t$ modulo $2$, such that $f(a_{1}+t \cdot 2^{3}) \equiv 0 \mod 2^{5}$. We can continue using Lemma \ref{Hensel} inductively until we find a $2$-adic integer $A = a + a_{1} + \ldots$ such that $f(A) = 0 \mod 2^{N}$ for all positive integers $N$. Thus, $\operatorname{-Id}$ is an element of all subgroups of $\mathcal{N}_{-7,0}(2^{\infty})$ of index $2$. Thus, quadratic twisting does not affect $\rho_{E,2^{\infty}}(G_{\Q})$. By Corollary \ref{subgroups of index 2 contain -Id}, if $\textit{j}_{E} = 16581375$, then $\rho_{E,2^{\infty}}(G_{\Q})$ is conjugate to $\mathcal{N}_{-7,0}(2^{\infty})$.

Let $E'/\QQ$ be an elliptic curve with $\textit{j}_{E'} = -3375$. Then $E'$ is $2$-isogenous to an elliptic curve with \textit{j}-invariant equal to $16581375$. By Corollary \ref{contains -Id}, $\rho_{E',2^{\infty}}(G_{\Q})$ contains $\operatorname{-Id}$ and hence, $\rho_{E',2^{\infty}}(G_{\Q})$ is not affected by quadratic twisting. Using code provided by Lozano-Robledo, we see that $E'$ has complex multiplication by an order of $K = \QQ(\sqrt{-7})$ with discriminant $\Delta_{K} = -7$ and conductor $f = 1$. By the fact that $\Delta_{K} \cdot f^{2} \equiv 1 \mod 4$, we let $\delta = \frac{\Delta_{K}-1}{4} \cdot f^{2} = -2$. Again, by the fact that $\Delta_{K} \cdot f^{2}$ is not divisible by $8$, Corollary \ref{Theorem 1.6 corollary 2} says that $\rho_{E',2^{\infty}}(G_{\Q})$ is conjugate to $\mathcal{N}_{-2,1}(2^{\infty})$.
\end{proof}

\begin{proposition}
Let $E/\QQ$ be an elliptic curve such that $E$ has CM. Suppose that the isogeny graph associated to the $\QQ$-isogeny class of $E$ is of $R_{4}(6)$ type. Then $\textit{j}_{E} = 0$ or $\textit{j}_{E} = 54000$. In the former case, $\rho_{E,2^{\infty}}(G_{\QQ})$ is conjugate to
$\left\langle \operatorname{-Id}, \begin{bmatrix} 0 & 1 \\ 1 & 0 \end{bmatrix}, \begin{bmatrix} 7 & 4 \\ -4 & 3 \end{bmatrix}, \begin{bmatrix} 3 & 6 \\ -6 & -3 \end{bmatrix} \right\rangle$ and in the latter case, $\rho_{E,2^{\infty}}(G_{\QQ})$ is conjugate to $\mathcal{N}_{-3,0}(2^{\infty})$.
\end{proposition}

\begin{proof}
The isogeny graph associated to the $\QQ$-isogeny class of $E$ is below 
\begin{center}
\begin{tikzcd}
{E_{1}, \textit{j}_{E_{1}} = 0} \arrow[dd, no head, "3"'] \arrow[rr, no head, "2"] &  & {E_{2}, \textit{j}_{E_{2}} = 54000} \arrow[dd, no head, "3"] \\
                                                                     &  &                                                        \\
{E_{3}, \textit{j}_{E_{3}} = 0} \arrow[rr, no head, "2"']                 &  & {E_{4}, \textit{j}_{E_{4}} = 54000}               
\end{tikzcd}
\end{center}
The elliptic curve $E$ has a point of order $2$ defined over $\QQ$. If $\textit{j}_{E} = 0$, then by Corollary \ref{points of order 2 with j=0}, $\rho_{E,2^{\infty}}(G_{\QQ})$ is conjugate to $\left\langle \operatorname{-Id}, \begin{bmatrix} 0 & 1 \\ 1 & 0 \end{bmatrix}, \begin{bmatrix} 7 & 4 \\ -4 & 3 \end{bmatrix}, \begin{bmatrix} 3 & 6 \\ -6 & -3 \end{bmatrix} \right\rangle$. If $\textit{j}_{E} = 54000$, then $E$ is $2$-isogenous to an elliptic curve $E'/\QQ$ such that $\textit{j}_{E} = 0$. Note that $\operatorname{-Id} \in \rho_{E',2^{\infty}}(G_{\QQ})$ and by Corollary \ref{contains -Id}, $\operatorname{-Id} \in \rho_{E,2^{\infty}}(G_{\QQ})$. Thus, quadratic twisting does not affect $\rho_{E,2^{\infty}}(G_{\QQ})$. 

Let $\widetilde{E}$ be the elliptic curve with LMFDB label \texttt{36.a1}. Then $\textit{j}_{\widetilde{E}} = 54000$. Running code provided by Lozano-Robledo, we see that $\widetilde{E}$ has CM by an order of $K = \QQ(\sqrt{-3})$ with discriminant $\Delta_{K} = -3$ and conductor $f = 2$. As $\Delta_{K} \cdot f^{2} \equiv 0 \mod 4$, we have $\delta = \frac{\Delta \cdot f^{2}}{4} = -3$ and $\phi = 0$. Note that $\Delta_{K} \cdot f^{2} = -12$ is not divisible by $8$. By Corollary \ref{Theorem 1.6 corollary 2}, $\rho_{\widetilde{E},2^{\infty}}(G_{\Q})$ is conjugate to $\mathcal{N}_{-3,0}(2^{\infty})$. As $E$ is a quadratic twist of $\widetilde{E}$, $\rho_{E,2^{\infty}}(G_{\QQ})$ is also conjugate to $\mathcal{N}_{-3,0}(2^{\infty})$.

\end{proof}

\begin{proposition}
Let $E/\QQ$ be an elliptic curve with $\textit{j}_{E} = 8000$. Then $E$ is $2$-isogenous to an elliptic curve $E'/\QQ$ with $\textit{j}_{E'} = 8000$. The isogeny graph associated to the $\QQ$-isogeny class of $E$ is of type $L_{2}(2)$. Denote
\begin{center} $H_{1,3} = \left\langle \begin{bmatrix} 1 & 0 \\ 0 & -1 \end{bmatrix}, \begin{bmatrix} 3 & 0 \\ 0 & 3 \end{bmatrix}, \begin{bmatrix} 1 & 1 \\ -2 & 1 \end{bmatrix} \right\rangle$ and $H_{-1,3} = \left\langle \begin{bmatrix} -1 & 0 \\ 0 & 1 \end{bmatrix}, \begin{bmatrix} 3 & 0 \\ 0 & 3 \end{bmatrix}, \begin{bmatrix} 1 & 1 \\ -2 & 1 \end{bmatrix} \right\rangle$. \end{center}
\end{proposition}
Then $\rho_{E,2^{\infty}}(G_{\QQ})$ fits in the following table
\begin{center} \begin{table}[h!]
 	\renewcommand{\arraystretch}{1.6}
	\begin{tabular} { |c|c|c| }
		\hline
		
		Isogeny graph & $\rho_{E_{1},2^{\infty}}(G_{\QQ})$ & $\rho_{E_{2},2^{\infty}}(G_{\QQ})$ \\
		\hline
		\multirow{2}*{\includegraphics[scale=0.05]{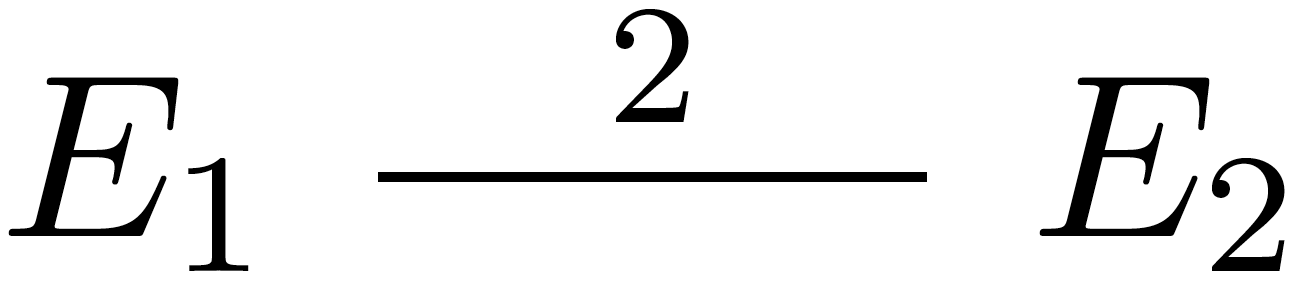}} & $H_{1, 3}$ & $H_{-1, 3}$ \\
		\cline{2-3}
		& $\mathcal{N}_{-2,0}(2^{\infty})$ & $\mathcal{N}_{-2,0}(2^{\infty})$ \\
		\hline
	\end{tabular}
\end{table} \end{center}

\begin{proof}

Let $E/\QQ$ be an elliptic curve such that $\textit{j}_{E} = 8000$. Then the isogeny graph associated to the $\QQ$-isogeny class of $E$ is of $L_{2}(2)$ type shown below:
\begin{center}
\begin{tikzcd}
E_{1} \arrow[rr, no head, "2"] &  & E_{2}
\end{tikzcd}
\end{center}
If $\rho_{E,2^{\infty}}(G_{\QQ})$ is not conjugate to $\mathcal{N}_{\delta,\phi}(2^{\infty})$, then by Theorem \ref{Theorem 1.6}, there are $\epsilon \in \{1, -1\}$ and $\alpha \in \{3, 5\}$, such that $\rho_{E,2^{\infty}}(G_{\QQ})$ is conjugate to
\begin{center} $H_{\epsilon,\alpha} = \left\langle \begin{bmatrix} \epsilon & 0 \\ 0 & -\epsilon \end{bmatrix}, \begin{bmatrix} \alpha & 0 \\ 0 & \alpha \end{bmatrix}, \begin{bmatrix} 1 & 1 \\ \delta & 1 \end{bmatrix} \right\rangle$ or $H_{\epsilon, \alpha}' = \left\langle \begin{bmatrix} \epsilon & 0 \\ 0 & -\epsilon \end{bmatrix}, \begin{bmatrix} \alpha & 0 \\ 0 & \alpha \end{bmatrix}, \begin{bmatrix} -1 & -1 \\ -\delta & -1 \end{bmatrix} \right\rangle$. \end{center}

Moreover, as $\textit{j}_{E} \neq 0, 1728$, all elliptic curves $E'/\QQ$ such that $\textit{j}_{E'} = 8000$ are quadratic twists of $E$. Let $E/\QQ$ be the elliptic curve with LMFDB label \texttt{256.a1}. Then $\textit{j}_{E} = 8000$. Using code provided by Lozano-Robledo, we see that $E$ has complex multiplication by an order of $K = \QQ(\sqrt{-2})$ with $\Delta_{K} = -8$ and conductor $f = 1$. We compute that $\delta = \frac{-8 \cdot 1^{2}}{4} = -2$ and thus, $\phi = 0$. .

A quick computation reveals that $H_{1,5}$, $H_{-1,5}$, $H_{1,5}'$, and $H_{-1,5}'$ are all equal to $\mathcal{N}_{-2,0}(2^{\infty})$ modulo $8$. By Corollary \ref{Theorem 1.6 corollary 1}, if $\overline{\rho}_{E,8}(G_{\QQ})$ is conjugate to $\mathcal{N}_{-2,0}(2^{\infty})$ modulo $8$, then $\rho_{E,2^{\infty}}(G_{\QQ})$ is conjugate to $\mathcal{N}_{-2,0}(2^{\infty})$. Another quick computation reveals that $H_{1,3}$ is conjugate to $H_{1,3}'$ and $H_{-1,3}$ is conjugate to $H_{-1,3}'$ modulo $8$. Neither $H_{1,3}$ nor $H_{-1,3}$ contain $\operatorname{-Id}$ and $H_{1,3}$ is not conjugate to $H_{-1,3}$. Moreover, $H_{1,3}$, $H_{-1,3}$, and $H_{1,5}$ are quadratic twists modulo $8$ with $H_{1,5} = \left\langle H_{1,3}, \operatorname{-Id} \right\rangle = \left\langle H_{-1,3}, \operatorname{-Id} \right\rangle$ modulo $8$. In other words, up to conjugation, there are three groups to work with $H_{1,3}$, $H_{-1,3}$, and $H_{1,5} = \mathcal{N}_{-2,0}(2^{\infty})$.

The elliptic curve $E_{1} : y^{2} = x^{3} - 17280x - 774144$ has LMFDB label \texttt{256.a1} and the elliptic curve $E_{2} : y^{2} = x^{3} - 4320x + 96768$ has LMFDB label \texttt{256.a2}. Moreover, $\textit{j}_{E_{1}} = \textit{j}_{E_{2}} = 8000$ and $E_{1}$ is $2$-isogenous to $E_{2}$. By part 2 of Example 9.4 in \cite{al-rCMGRs}, $\rho_{E_{1},2^{\infty}}(G_{\QQ})$ is conjugate to $H_{-1,3}$ and $\rho_{E_{2},2^{\infty}}(G_{\QQ})$ is conjugate to $H_{1,3}$.

Finally, let $E/\QQ$ be the elliptic curve with LMFDB label \texttt{2304.h1}. Then $\textit{j}_{E} = 8000$. Hence, the isogeny graph associated to the $\QQ$-isogeny class of $E$ is of $L_{2}(2)$ type and $E$ is $\QQ$-isogenous to one other elliptic curve $E'$ defined over $\QQ$. Using code provided by Lozano-Robledo, we see that $\overline{\rho}_{E,8}(G_{\QQ})$ contains $\operatorname{-Id}$ and by Lemma \ref{contains -Id}, $E'$ does too. By the fact that the \textit{j}-invariants of both $E$ and $E'$ equal $8000$ and $\overline{\rho}_{E,8}(G_{\QQ})$ and $\overline{\rho}_{E',8}(G_{\QQ})$ both contain $\operatorname{-Id}$, we have that $\rho_{E,2^{\infty}}(G_{\QQ})$ and $\rho_{E,2^{\infty}}(G_{\QQ})$ are both conjugate to $\mathcal{N}_{-2,0}(2^{\infty})$.
\end{proof}

\begin{proposition}
Define the following subgroups of $\operatorname{GL}(2, \ZZ_{2})$
\begin{itemize}
    \item $G_{1} = \left\{ \begin{bmatrix} a & b \\ -b & a \end{bmatrix} : a^{2}+b^{2} \not \equiv 0 \mod 2 \right\}$
    \item $G_{2,a} = \left\langle \operatorname{-Id}, 3 \cdot \operatorname{Id}, \begin{bmatrix} 1 & 2 \\ -2 & 1 \end{bmatrix} \right\rangle$
    \item $G_{2,b} = \left\langle \operatorname{-Id}, 3 \cdot \operatorname{Id}, \begin{bmatrix} 2 & 1 \\ -1 & 2 \end{bmatrix} \right\rangle$
    \item $G_{4,a} = \left\langle 5 \cdot \operatorname{Id}, \begin{bmatrix} 1 & 2 \\ -2 & 1 \end{bmatrix} \right\rangle$
    \item $G_{4,b} = \left\langle 5 \cdot \operatorname{Id}, \begin{bmatrix} -1 & -2 \\ 2 & -1 \end{bmatrix} \right\rangle$
    \item $G_{4,c} = \left\langle 3 \cdot \operatorname{Id}, \begin{bmatrix} 2 & -1 \\ 1 & 2 \end{bmatrix} \right\rangle$
    \item $G_{4,d} = \left\langle 3 \cdot \operatorname{Id}, \begin{bmatrix} -2 & 1 \\ -1 & -2 \end{bmatrix} \right\rangle$
\end{itemize}
and let $\Gamma = \left\{c_{1} = \begin{bmatrix} 1 & 0 \\ 0 & -1 \end{bmatrix}, c_{-1} = \begin{bmatrix} -1 & 0 \\ 0 & 1 \end{bmatrix}, c_{1}' = \begin{bmatrix} 0 & 1 \\ 1 & 0 \end{bmatrix}, c_{-1}' = \begin{bmatrix} 0 & -1 \\ -1 & 0 \end{bmatrix} \right\}$. Let $E/\QQ$ be an elliptic curve such that $\textit{j}_{E} = 1728$. Then $\rho_{E,2^{\infty}}(G_{\QQ})$ is conjugate to $\left\langle H, \gamma \right\rangle$ where $H$ is one of the seven groups above and $\gamma \in \Gamma$. Either the isogeny graph associated to the $\QQ$-isogeny class $E$ is of type $T_{4}$ or the isogeny graph associated to the $\QQ$-isogeny class of $E$ is of type $L_{2}(2)$. For $\epsilon \in \left\{\pm 1 \right\}$, denote
\begin{center} $H_{\epsilon} = \left\langle \begin{bmatrix} \epsilon & 0 \\ 0 & -\epsilon \end{bmatrix}, \begin{bmatrix} 5 & 0 \\ 0 & 5 \end{bmatrix}, \begin{bmatrix} 1 & 1 \\ -4 & 1 \end{bmatrix} \right\rangle$ and $H_{\epsilon}' = \left\langle \begin{bmatrix} \epsilon & 0 \\ 0 & -\epsilon \end{bmatrix}, \begin{bmatrix} 5 & 0 \\ 0 & 5 \end{bmatrix}, \begin{bmatrix} -1 & -1 \\ 4 & -1 \end{bmatrix} \right\rangle$. \end{center}

If the isogeny graph associated to the $\QQ$-isogeny class of $E$ is of $T_{4}$ type, then the $\QQ$-isogeny class of $E$ consists of four elliptic curves over $\QQ$, $E_{1}$, $E_{2}$, $E_{3}$, and $E_{4}$, such that $\textit{j}_{E_{1}} = \textit{j}_{E_{2}} = 1728$ and $\textit{j}_{E_{3}} = \textit{j}_{E_{4}} = 287496$ with the following algebraic data:
\begin{center} \begin{table}[h!]
 	\renewcommand{\arraystretch}{1.6}
 	\scalebox{1}{
	\begin{tabular}{|c|c|c|c|c|c|}
		\hline
		Isogeny graph & Torsion Configuration & $\rho_{E_{1},2^{\infty}}(G_{\QQ})$ & $\rho_{E_{2},2^{\infty}}(G_{\QQ})$ & $\rho_{E_{3},2^{\infty}}(G_{\QQ})$ & $\rho_{E_{4},2^{\infty}}(G_{\QQ})$ \\
		\hline
		\multirow{3}*{\includegraphics[scale=0.06]{T4_isogeny_graph.png}} & $([2,2],[2],[2],[2])$ & $\left\langle G_{2,a}, c_{1} \right\rangle$ & $\left\langle G_{2,a}, c_{1}' \right\rangle$ & $\mathcal{N}_{-4,0}(2^{\infty})$ & $\mathcal{N}_{-4,0}(2^{\infty})$ \\
		\cline{2-6}
		& $([2,2],[2],[4],[2])$ & $\left\langle G_{4,a}, c_{1} \right\rangle$ & $\left\langle G_{4,a}, c_{1}' \right\rangle$ & $H_{1}$ & $H_{-1}$ \\
		\cline{2-6}
		& $([2,2],[4],[4],[2])$ & $\left\langle G_{4,b}, c_{1} \right\rangle$ & $\left\langle G_{4,b}, c_{1}' \right\rangle$ & $H_{1}'$ & $H_{-1}'$ \\
		\hline
	\end{tabular}}
\end{table} \end{center}
If the isogeny graph associated to the $\QQ$-isogeny class of $E$ is of $L_{2}(2)$ type, then, the $\QQ$-isogeny class of $E$ consists of two elliptic curves over $\QQ$, $E_{1}$ and $E_{2}$ such that $\textit{j}_{E_{1}} = \textit{j}_{E_{2}} = 1728$ with the following algebraic data.
\begin{center} \begin{table}[h!]
 	\renewcommand{\arraystretch}{1.6}
 	\scalebox{1}{
	\begin{tabular}{|c|c|c|c|}
		\hline
		Isogeny graph & Torsion configuration & $\rho_{E_{1},2^{\infty}}(G_{\QQ})$ & $\rho_{E_{2},2^{\infty}}(G_{\QQ})$ \\
		\hline
		\multirow{4}*{\includegraphics[scale=0.09]{L22_graph.png}} & \multirow{4}*{$([2],[2])$} & $\left\langle G_{2,b}, c_{1} \right\rangle$ & $\left\langle G_{2,b}, c_{1}' \right\rangle$ \\
		\cline{3-4}
		& & $\left\langle G_{4,c}, c_{1} \right\rangle$ & $\left\langle G_{4,c}, c_{1}' \right\rangle$ \\
		\cline{3-4}
		& & $\left\langle G_{4,d}, c_{1} \right\rangle$ & $\left\langle G_{4,d}, c_{1}' \right\rangle$ \\
		\cline{3-4}
		& & $\mathcal{N}_{-1,0}(2^{\infty})$ & $\mathcal{N}_{-1,0}(2^{\infty})$ \\
		\hline
	\end{tabular}}
\end{table} \end{center}
\end{proposition}

\begin{proof}
Let $E/\QQ$ be an elliptic curve with $\textit{j}_{E} = 1728$. Then $E$ has CM by an order of $K = \QQ(i)$ with discriminant $\Delta_{K} = -4$ and conductor $f = 1$. Then $\delta = \frac{\Delta_{K} \cdot f^{2}}{4} = -1$. By Theorem \ref{Theorem 1.1}, $\rho_{E,2^{\infty}}(G_{\QQ})$ is conjugate to a subgroup of $\mathcal{N}_{-1,0}(2^{\infty}) = \left\langle G_{1}, \begin{bmatrix} -1 & 0 \\ 0 & 1 \end{bmatrix} \right\rangle$. More precisely, by Theorem \ref{Theorem 1.7}, $\rho_{E,2^{\infty}}(G_{\Q})$ is generated by one of $G_{1}$, $G_{2,a}$, $G_{2b}$, $G_{4,a}$, $G_{4,b}$, $G_{4,c}$, $G_{4,d}$, and one element of $\Gamma$.

First note that $$\begin{bmatrix} 0 & 1 \\ -1 & 0 \end{bmatrix} \cdot \begin{bmatrix} 1 & 0 \\ 0 & -1 \end{bmatrix} = \begin{bmatrix} 0 & -1 \\ -1 & 0 \end{bmatrix}.$$
Let $\gamma, \gamma' \in \Gamma$. By the fact that $G_{1}$ contains $\operatorname{-Id}$ and $\begin{bmatrix} 0 & 1 \\ -1 & 0 \end{bmatrix}$, we have that $\left\langle G_{1}, \gamma \right\rangle = \left\langle G_{1}, \gamma' \right\rangle = \mathcal{N}_{-1,0}(2^{\infty})$. By the fact that $G_{2,a}$ and $G_{2,b}$ contain $\operatorname{-Id}$, $\left\langle G_{2,a}, c_{1} \right\rangle = \left\langle G_{2,a}, c_{-1} \right\rangle$ and $\left\langle G_{2,b}, c_{1} \right\rangle = \left\langle G_{2,b}, c_{-1} \right\rangle$. Similarly, $\left\langle G_{2,a}, c_{1}' \right\rangle = \left\langle G_{2,a}, c_{-1}' \right\rangle$ and $\left\langle G_{2,b}, c_{1}' \right\rangle = \left\langle G_{2,b}, c_{-1}' \right\rangle$. Note that
$$\left\langle G_{4,a}, c_{1} \right\rangle = \left\langle 5 \cdot \operatorname{Id}, \begin{bmatrix} 1 & 2 \\ -2 & 1 \end{bmatrix}, c_{1} \right\rangle = \left\langle 5 \cdot \operatorname{Id}, c_{1} \cdot \begin{bmatrix} 1 & 2 \\ -2 & 1 \end{bmatrix} \cdot c_{1}^{-1}, c_{1} \right\rangle = \left\langle 5 \cdot \operatorname{Id}, \begin{bmatrix} 1 & -2 \\ 2 & 1 \end{bmatrix}, c_{1} \right\rangle.$$ After switching the orders of the generators of $\left\langle G_{4,a}, c_{1} \right\rangle$, we see that this last group is conjugate to the group $\left\langle 5 \cdot \operatorname{Id}, \begin{bmatrix} 1 & 2 \\ -2 & 1  \end{bmatrix}, c_{-1} \right\rangle = \left\langle G_{4,a}, c_{-1} \right\rangle$. Hence, $\left\langle G_{4,a}, c_{1} \right\rangle$ is conjugate to $\left\langle G_{4,a}, c_{-1} \right\rangle$.

Next, note that $\left\langle G_{4,a}, c_{1}' \right\rangle = \left\langle 5 \cdot \operatorname{Id}, \begin{bmatrix} 1 & 2 \\ -2 & 1
\end{bmatrix}, \begin{bmatrix} 0 & 1 \\ 1 & 0 \end{bmatrix} \right\rangle$ is conjugate to the group
$$\left\langle 5 \cdot \operatorname{Id}, c_{1} \cdot \begin{bmatrix} 1 & 2 \\ -2 & 1 \end{bmatrix} \cdot c_{1}^{-1}, c_{1} \cdot \begin{bmatrix} 0 & 1 \\ 1 & 0 \end{bmatrix} \cdot c_{1}^{-1} \right\rangle = \left\langle 5 \cdot \operatorname{Id}, \begin{bmatrix} 1 & -2 \\ 2 & 1 \end{bmatrix}, \begin{bmatrix} 0 & -1 \\ -1 & 0 \end{bmatrix} \right\rangle.$$
After switching the order of the generators, this last group is conjugate to $\left\langle 5 \cdot \operatorname{Id}, \begin{bmatrix} 1 & 2 \\ -2 & 1 \end{bmatrix}, \begin{bmatrix} 0 & -1 \\ -1 & 0 \end{bmatrix} \right\rangle = \left\langle G_{4,a}, c_{-1}' \right\rangle$. Hence, $\left\langle G_{4,a}, c_{1}' \right\rangle$ is conjugate to $\left\langle G_{4,a}, c_{-1}' \right\rangle$.

Making similar computations, we see that $\left\langle G_{4,x}, c_{1} \right\rangle$ is conjugate to $\left\langle G_{4,x}, c_{-1} \right\rangle$ and $\left\langle G_{4,x}, c_{1}' \right\rangle$ is conjugate to $\left\langle G_{4,x}, c_{-1}' \right\rangle$ for $x = b$, $c$, and $d$. Hence, we work with the case that $\rho_{E,2^{\infty}}(G_{\Q})$ is conjugate to $\mathcal{N}_{-1,0}(2^{\infty})$, $\left\langle G_{2,a}, c_{1} \right\rangle$, $\left\langle G_{2,a}, c_{1}' \right\rangle$, $\left\langle G_{2,b}, c_{1} \right\rangle$, $\left\langle G_{2,b}, c_{1}' \right\rangle$, $\left\langle G_{4,a}, c_{1} \right\rangle$, $\left\langle G_{4,a}, c_{1}' \right\rangle$, $\left\langle G_{4,b}, c_{1} \right\rangle$, $\left\langle G_{4,b}, c_{1}' \right\rangle$, $\left\langle G_{4,c}, c_{1} \right\rangle$, $\left\langle G_{4,c}, c_{1}' \right\rangle$, $\left\langle G_{4,d}, c_{1} \right\rangle$, or $\left\langle G_{4,d}, c_{1}' \right\rangle$.

If $\rho_{E,2^{\infty}}(G_{\Q})$ is conjugate to $\left\langle G_{2,a}, c_{1} \right\rangle$, $\left\langle G_{4,a}, c_{1} \right\rangle$, or $\left\langle G_{4,b}, c_{1} \right\rangle$, then $\rho_{E,2^{\infty}}(G_{\Q})$ reduces modulo $2$ to the trivial group and hence, $E$ has full two-torsion defined over $\QQ$. Moreover, $\left\langle G_{2,a}, c_{1} \right\rangle$ is the only group of those three which contains $\operatorname{-Id}$ and $\left\langle G_{2,a}, c_{1} \right\rangle = \left\langle G_{4,a}, c_{1}, \operatorname{-Id} \right\rangle = \left\langle G_{4,b}, c_{1}, \operatorname{-Id} \right\rangle$. In other words, $\left\langle G_{2,a}, c_{1} \right\rangle$, $\left\langle G_{4,a}, c_{1} \right\rangle$, and $\left\langle G_{4,b}, c_{1} \right\rangle$ are quadratic twists.

Next, if $\rho_{E,2^{\infty}}(G_{\Q})$ is conjugate to one of $\left\langle G_{2,a}, c_{1}' \right\rangle$, $\left\langle G_{4,a}, c_{1}' \right\rangle$, or $\left\langle G_{4,b}, c_{1}' \right\rangle$, then using magma \cite{magma}, we see that $\rho_{E,2^{\infty}}(G_{\Q})$ is conjugate to a subgroup of $\left\{ \begin{bmatrix} \ast & \ast \\ 0 & \ast \end{bmatrix} \right\} \subseteq \operatorname{GL}(2, \ZZ / 4 \ZZ)$. Thus, $E$ has a cyclic, $\QQ$-rational subgroup of order $4$ that generates a $\QQ$-rational $4$-isogeny with cyclic kernel. Moreover, $\left\langle G_{2,a}, c_{1}' \right\rangle$ is the only group of those three which contains $\operatorname{-Id}$ and $\left\langle G_{2,a}, c_{1}' \right\rangle = \left\langle G_{4,a}, c_{1}', \operatorname{-Id} \right\rangle = \left\langle G_{4,b}, c_{1}', \operatorname{-Id} \right\rangle$. In other words, $\left\langle G_{2,a}, c_{1}' \right\rangle$, $\left\langle G_{4,a}, c_{1}' \right\rangle$, and $\left\langle G_{4,b}, c_{1}' \right\rangle$ are quadratic twists.

On the other hand, if $\rho_{E,2^{\infty}}(G_{\Q})$ is conjugate to any one of the remaining seven groups, $\mathcal{N}_{-1,0}(2^{\infty})$, $\left\langle G_{2,b}, c_{1} \right\rangle$, $\left\langle G_{2,b}, c_{1}' \right\rangle$, $\left\langle G_{4,c}, c_{1} \right\rangle$, $\left\langle G_{4,c}, c_{1}' \right\rangle$, $\left\langle G_{4,d}, c_{1} \right\rangle$, or $\left\langle G_{4,d}, c_{1}' \right\rangle$ then $E(\QQ)_{\texttt{tors}} \cong \ZZ / 2 \ZZ$ and $E$ does not have a cyclic, $\QQ$-rational subgroup of order $4$. Of those seven groups $\mathcal{N}_{-1,0}(2^{\infty})$, $\left\langle G_{2,b}, c_{1} \right\rangle$, and $\left\langle G_{2,b}, c_{1}' \right\rangle$ contain $\operatorname{-Id}$. Moreover, $\left\langle G_{2,b}, c_{1} \right\rangle = \left\langle G_{4,c}, c_{1}, \operatorname{-Id} \right\rangle = \left\langle G_{4,d}, c_{1}, \operatorname{-Id} \right\rangle$ and $\left\langle G_{2,b}, c_{1}' \right\rangle = \left\langle G_{4,c}, c_{1}', \operatorname{-Id} \right\rangle = \left\langle G_{4,d}, c_{1}', \operatorname{-Id} \right\rangle$. In other words, $\left\langle G_{2,b}, c_{1} \right\rangle$, $\left\langle G_{4,c}, c_{1} \right\rangle$, and $\left\langle G_{4,d}, c_{1} \right\rangle$ are quadratic twists and $\left\langle G_{2,b}, c_{1}' \right\rangle$, $\left\langle G_{4,c}, c_{1}' \right\rangle$, and $\left\langle G_{4,d}, c_{1}' \right\rangle$ are quadratic twists.

First, we will find examples of elliptic curves over $\QQ$ whose $2$-adic Galois image is conjugate to $\left\langle G_{2,a}, c_{1}' \right\rangle$, $\left\langle G_{4,a}, c_{1}' \right\rangle$, and $\left\langle G_{4,b}, c_{1}' \right\rangle$; the groups that can serve as the $2$-adic Galois image attached to an elliptic curve over $\QQ$ with \textit{j}-invariant equal to $1728$ with a cyclic, $\QQ$-rational subgroup of order $4$, and then classify the $2$-adic Galois image of the elliptic curves in their $\QQ$-isogeny classes. In this case, the isogeny graph associated to the $\QQ$-isogeny class is of $T_{4}$ type and $E$ is represented by the elliptic curve labeled $E_{2}$ (see below):
\begin{center}
\begin{tikzcd}
                & E_{2}                                                 &                 \\
                & E_{1} \arrow[u, no head, "2"] \arrow[ld, no head, "2"'] \arrow[rd, no head, "2"] &                 \\
E_{3} &                                                                 & E_{4}
\end{tikzcd}
\end{center}
\begin{itemize}
    \item $\left\langle G_{2,a}, c_{1}' \right\rangle$
    
    Example 9.8 in \cite{al-rCMGRs} says that for $E = E_{2} : y^{2} = x^{3}+9x$, $\textit{j}_{E} = 1728$, and $\rho_{E,2^{\infty}}(G_{\QQ})$ is conjugate to $\left\langle G_{2,a}, c_{1}' \right\rangle$. The $\QQ$-isogeny class of $E$ has LMFDB label \texttt{576.c}. The isogeny-torsion graph associated to \texttt{576.c} is of type $T_{4}$ with torsion configuration $([2,2],[2],[2],[2])$. All elliptic curves in the $\QQ$-isogeny class of $E$ have CM by an order of $K = \QQ(i)$ with discriminant $\Delta_{K} = -4$.
    
    \begin{enumerate}
    
    \item The elliptic curve $E_{1}/\QQ$ with LMFDB label \texttt{576.c3} is $2$-isogenous to $E = E_{2}$. Moreover, $\textit{j}_{E_{1}} = 1728$ and $E_{1}(\QQ)_{\text{tors}} \cong \ZZ / 2 \ZZ \times \ZZ / 2 \ZZ$. By Lemma \ref{contains -Id}, $\rho_{E_{1},2^{\infty}}(G_{\QQ})$ contains $\operatorname{-Id}$ and hence, $\rho_{E_{1},2^{\infty}}(G_{\QQ})$ is conjugate to $\left\langle G_{2,a}, c_{1} \right\rangle$.
    
    \item The elliptic curve $E_{3} / \QQ$ with LMFDB label \texttt{576.c1} is $4$-isogenous to $E = E_{2}$ and $\textit{j}_{E_{3}} = 287496$. Using code provided by Lozano-Robledo, we see that $E_{3}$ has CM by an order of $K = \QQ(i)$ with discriminant $\Delta_{K} = -4$ and conductor $f = 2$. Hence, $\delta = \frac{\Delta_{K} \cdot f^{2}}{4} = -4$ and $\phi = 0$. By Theorem \ref{Theorem 1.2}, $\rho_{E_{3},2^{\infty}}(G_{\QQ})$ is contained in $\mathcal{N}_{-4,0}(2^{\infty})$. Moreover, $\rho_{E_{3},2^{\infty}}(G_{\QQ})$ is a group of level $4$ and $\overline{\rho}_{E_{3},4}(G_{\QQ})$ is a group of order $16$ in $\operatorname{GL}(2, \ZZ / 4 \ZZ)$. The reduction of $\mathcal{N}_{-4,0}(2^{\infty})$ modulo $4$ is a group of order $16$. In other words, $\rho_{E_{3},2^{\infty}}(G_{\QQ})$ is conjugate to $\mathcal{N}_{-4,0}(2^{\infty})$.
    
    \item The elliptic curve $E_{4} / \QQ$ with LMFDB label \texttt{576.c2} is $4$-isogenous to $E = E_{2}$. Moreover, $\textit{j}_{E_{4}} = 287496$. Thus, it is a quadratic twist of $E_{3}$. By Corollary \ref{contains -Id}, $\rho_{E_{4},2^{\infty}}(G_{\QQ})$ contains $\operatorname{-Id}$. Thus, $\rho_{E_{4},2^{\infty}}(G_{\QQ})$ is also conjugate to $\mathcal{N}_{-4,0}(2^{\infty})$ as the only quadratic twist of $\mathcal{N}_{-4,0}(2^{\infty})$ that contains $\operatorname{-Id}$ is $\mathcal{N}_{-4,0}(2^{\infty})$ itself.
    \end{enumerate}
    
    \item $\left\langle G_{4,a}, c_{1}' \right\rangle$
    
    Example 9.8 in \cite{al-rCMGRs} says that for $E = E_{2} : y^{2} = x^{3} + x$, $\textit{j}_{E} = 1728$, and $\rho_{E,2^{\infty}}(G_{\QQ})$ is conjugate to $\left\langle G_{4,a}, c_{1}' \right\rangle$. The $\QQ$-isogeny class of $E$ has LMFDB label \texttt{64.a}. The isogeny-torsion graph associated to \texttt{64.a} is of type $T_{4}$ with torsion configuration $([2,2],[2],[4],[2])$ (note that in this case, one of the elliptic curves with $\textit{j}$-invariant = $287496$ has a point of order $4$ defined over $\QQ$). Note that $E(\QQ)_{\text{tors}} \cong \ZZ / 2 \ZZ$. All elliptic curves in the $\QQ$-isogeny class of $E$ have CM by an order of $K = \QQ(i)$ with discriminant $\Delta_{K} = -4$.
    
    \begin{enumerate}
        \item The elliptic curve $E_{1}/\QQ$ with LMFDB label \texttt{64.a3} is $2$-isogenous to $E = E_{2}$. Moreover, $E_{1}$ is isomorphic to the elliptic curve $y^{2} = x^{3} - 4x$, $\textit{j}_{E_{1}} = 1728$, and $E_{1}(\QQ)_{\text{tors}} \cong \ZZ / 2 \ZZ \times \ZZ / 2 \ZZ$. By Example 9.8 in \cite{al-rCMGRs}, $\rho_{E_{1},2^{\infty}}(G_{\QQ})$ is conjugate to $\left\langle G_{4,a}, c_{1} \right\rangle$.
        
        \item The elliptic curve $E_{3} / \QQ$ with LMFDB label \texttt{64.a2} is $4$-isogenous to $E = E_{2}$. Moreover, $E_{3}$ is isomorphic to the elliptic curve $y^{2}=x^{3}-44x+112$, $\textit{j}_{E_{3}} = 287496$, and $E_{3}(\QQ)_{\text{tors}} \cong \ZZ / 4 \ZZ$. The elliptic curve $E_{3}$ has CM by an order of $K = \QQ(i)$ with $\Delta_{K} = -4$ and conductor $f = 2$. Thus, $\delta = \frac{\Delta_{K} \cdot f^{2}}{4} = -4$. By Example 9.4 in \cite{al-rCMGRs}, $\rho_{E_{3},2^{\infty}}(G_{\QQ})$ is conjugate to $\left\langle \begin{bmatrix} 1 & 0 \\ 0 & -1 \end{bmatrix}, \begin{bmatrix} 5 & 0 \\ 0 & 5 \end{bmatrix}, \begin{bmatrix} 1 & 1 \\ -4 & 1 \end{bmatrix} \right\rangle \subseteq \mathcal{N}_{-4,0}(2^{\infty}) \subseteq \operatorname{GL}(2, \ZZ_{2})$.
        
        \item The elliptic curve $E_{4} / \QQ$ with LMFDB label \texttt{64.a1} is $4$-isogenous to $E = E_{2}$. Moreover, $E_{4}$ is isomorphic to the elliptic curve $y^{2}=x^{3}-44x-112$, $\textit{j}_{E_{4}} = 287496$, and $E_{4}(\QQ)_{\text{tors}} \cong \ZZ / 2 \ZZ$. The elliptic curve $E_{4}$ has CM by an order of $K = \QQ(i)$ with $\Delta_{K} = -4$ and conductor $f = 2$. Thus, $\delta = \frac{\Delta_{K} \cdot f^{2}}{4} = -4$. By Example 9.4 in \cite{al-rCMGRs}, $\rho_{E_{4},2^{\infty}}(G_{\QQ})$ is conjugate to $\left\langle \begin{bmatrix} -1 & 0 \\ 0 & 1 \end{bmatrix}, \begin{bmatrix} 5 & 0 \\ 0 & 5 \end{bmatrix}, \begin{bmatrix} 1 & 1 \\ -4 & 1 \end{bmatrix} \right\rangle \subseteq \mathcal{N}_{-4,0}(2^{\infty}) \subseteq \operatorname{GL}(2, \ZZ_{2})$.
    \end{enumerate}
    
    \item $\left\langle G_{4,b}, c_{1}' \right\rangle$
    
    Example 9.8 in \cite{al-rCMGRs} says that for $E = E_{2} : y^{2} = x^{3} + 4x$, $\textit{j}_{E} = 1728$, and $\rho_{E,2^{\infty}}(G_{\QQ})$ is conjugate to $\left\langle G_{4,b}, c_{1}' \right\rangle$. The $\QQ$-isogeny class of $E$ has LMFDB label \texttt{32.a}. The isogeny-torsion graph associated to \texttt{32.a} is of type $T_{4}$ with torsion configuration $([2,2],[4],[4],[2])$. Note that $E(\QQ)_{\text{tors}} \cong \ZZ / 4 \ZZ$. All elliptic curves in the $\QQ$-isogeny class of $E$ have CM by an order of $K = \QQ(i)$ with discriminant $\Delta_{K} = -4$.
    
    \begin{enumerate}
        \item The elliptic curve $E_{1}/\QQ$ with LMFDB label \texttt{32.a3} is $2$-isogenous to $E = E_{2}$. Moreover, $E_{1}$ is isomorphic to the elliptic curve $y^{2} = x^{3} - x$, $\textit{j}_{E_{1}} = 1728$, and $E_{1}(\QQ)_{\text{tors}} \cong \ZZ / 2 \ZZ \times \ZZ / 2 \ZZ$. By Example 9.8 in \cite{al-rCMGRs}, $\rho_{E_{1},2^{\infty}}(G_{\QQ})$ is conjugate to $\left\langle G_{4,b}, c_{1} \right\rangle$.
        
        \item The elliptic curve $E_{3} / \QQ$ with LMFDB label \texttt{32.a2} is $4$-isogenous to $E = E_{2}$. Moreover, $E_{3}$ is isomorphic to the elliptic curve $y^{2}=x^{3}-11x+14$, $\textit{j}_{E_{3}} = 287496$, and $E_{3}(\QQ)_{\text{tors}} \cong \ZZ / 4 \ZZ$. Moreover, $E_{3}$ has CM by an order of $K = \QQ(i)$ with $\Delta_{K} = -4$ and conductor $f = 2$. Thus, $\delta = \frac{\Delta_{K} \cdot f^{2}}{4} = -4$. By Example 9.4 in \cite{al-rCMGRs}, $\rho_{E_{3},2^{\infty}}(G_{\QQ})$ is conjugate to $\left\langle \begin{bmatrix} 1 & 0 \\ 0 & -1 \end{bmatrix}, \begin{bmatrix} 5 & 0 \\ 0 & 5 \end{bmatrix}, \begin{bmatrix} -1 & -1 \\ 4 & -1 \end{bmatrix} \right\rangle \subseteq \mathcal{N}_{-4,0}(2^{\infty}) \subseteq \operatorname{GL}(2, \ZZ_{2})$.
        
        \item The elliptic curve $E_{4} / \QQ$ with LMFDB label \texttt{32.a1} is $4$-isogenous to $E = E_{2}$. Moreover, $E_{4}$ is isomorphic to the elliptic curve $y^{2}=x^{3}-11x-14$, $\textit{j}_{E_{4}} = 287496$, and $E_{4}(\QQ)_{\text{tors}} \cong \ZZ / 2 \ZZ$. The elliptic curve $E_{4}$ has CM by an order of $K = \QQ(i)$ with $\Delta_{K} = -4$ and conductor $f = 2$. Thus, $\delta = \frac{\Delta_{K} \cdot f^{2}}{4} = -4$. By Example 9.4 in \cite{al-rCMGRs}, $\rho_{E_{4},2^{\infty}}(G_{\QQ})$ is conjugate to the group $\left\langle \begin{bmatrix} -1 & 0 \\ 0 & 1 \end{bmatrix}, \begin{bmatrix} 5 & 0 \\ 0 & 5 \end{bmatrix}, \begin{bmatrix} -1 & -1 \\ 4 & -1 \end{bmatrix} \right\rangle \subseteq \mathcal{N}_{-4,0}(2^{\infty}) \subseteq \operatorname{GL}(2, \ZZ_{2})$.
    \end{enumerate}
    
    Now we move on to classify the $2$-adic Galois image of isogeny-torsion graphs of $L_{2}(2)$ type with CM whose elliptic curves have \textit{j}-invariant equal to $1728$. Up to conjugation, there are seven possible $2$-adic Galois images, $\mathcal{N}_{-1,0}(2^{\infty})$, $\left\langle G_{2,b}, c_{1} \right\rangle$, $\left\langle G_{2,b}, c_{1}' \right\rangle$, $\left\langle G_{4,c}, c_{1} \right\rangle$, $\left\langle G_{4,c}, c_{1}' \right\rangle$, $\left\langle G_{4,d}, c_{1} \right\rangle$, and $\left\langle G_{4,d}, c_{1}' \right\rangle$. We will prove that there are four distinct arrangements. In this case, the $\QQ$-isogeny class of $E$ has two curves, both elliptic curves have \textit{j}-invariant equal to $1728$, and the isogeny graph associated to the $\QQ$-isogeny class of $E$ is of $L_{2}(2)$ type (see below):
    \begin{center}
    \begin{tikzcd}
E_{1} \arrow[r, "2", no head] & E_{2}
\end{tikzcd}
\end{center}
    \begin{enumerate}
        \item $\left\langle G_{2,b}, c_{-1} \right\rangle$ and $\left\langle G_{2,b}, c_{-1}' \right\rangle$
        
        Let $E_{1} : y^{2} = x^{3} + 18x$, let $E_{2} : y^{2} = x^{3} - 72x$, let $E_{1}' : y^{2} = x^{3} - 18x$, and let $E_{2}' : y^{2} = x^{3}+72x$. Then $E_{1}$ is $2$-isogenous to $E_{2}$ and $E_{1}'$ is $2$-isogenous to $E_{2}'$. By Example 9.8 in \cite{al-rCMGRs}, $\rho_{E_{1},2^{\infty}}(G_{\QQ})$ is conjugate to $\left\langle G_{2,b}, c_{-1}' \right\rangle$ and $\rho_{E_{1}',2^{\infty}}(G_{\QQ})$ is conjugate to $\left\langle G_{2,b}, c_{-1} \right\rangle$.
        
        We claim that $\rho_{E_{2},2^{\infty}}(G_{\QQ})$ is conjugate to $\left\langle G_{2,b}, c_{-1} \right\rangle$ and $\rho_{E_{2}',2^{\infty}}(G_{\QQ})$ is conjugate to $\left\langle G_{2,b}, c_{-1}' \right\rangle$. Note that $E_{1}$ is a quadratic twist of $E_{2}'$ (by $2$) and $E_{2}$ is a quadratic twist of $E_{1}'$ (by $2$). By Corollary \ref{contains -Id}, $\rho_{E_{2},2^{\infty}}(G_{\QQ})$ and $\rho_{E_{2}',2^{\infty}}(G_{\QQ})$ contain $\operatorname{-Id}$. The only quadratic twist of $\left\langle G_{2,b}, c_{-1} \right\rangle$ that contains $\operatorname{-Id}$ is $\left\langle G_{2,b}, c_{-1} \right\rangle$ itself and the only quadratic twist of $\left\langle G_{2,b}, c_{-1}' \right\rangle$ that contains $\operatorname{-Id}$ is $\left\langle G_{2,b}, c_{-1}' \right\rangle$ itself.
        
        \item $\left\langle G_{4,c}, c_{1} \right\rangle$ and $\left\langle G_{4,c}, c_{1}' \right\rangle$
        
        Let $E_{1} : y^{2} = x^{3} + 2x$ and let $E_{2} : y^{2} = x^{3} - 8x$. Then $E_{1}$ is $2$-isogenous to $E_{2}$. By Example 9.8 in \cite{al-rCMGRs}, $\rho_{E_{1},2^{\infty}}(G_{\QQ})$ is conjugate to $\left\langle G_{4,c}, c_{1}' \right\rangle$ and $\rho_{E_{2},2^{\infty}}(G_{\QQ})$ is conjugate to $\left\langle G_{4,c}, c_{1} \right\rangle$.
        
        \item $\left\langle G_{4,d}, c_{1} \right\rangle$ and $\left\langle G_{4,d}, c_{1}' \right\rangle$
        
        Let $E_{1} : y^{2} = x^{3} - 2x$ and let $E_{2} : y^{2} = x^{3} + 8x$. Then $E_{1}$ is $2$-isogenous to $E_{2}$. By Example 9.8 in \cite{al-rCMGRs}, $\rho_{E_{1},2^{\infty}}(G_{\QQ})$ is conjugate to $\left\langle G_{4,d}, c_{1} \right\rangle$ and $\rho_{E_{2},2^{\infty}}(G_{\QQ})$ is conjugate to $\left\langle G_{4,d}, c_{1}' \right\rangle$.
        
        \item $\mathcal{N}_{-1,0}(2^{\infty})$
        
        Let $E_{1}/\QQ$ be an elliptic curve with $\textit{j}_{E_{1}} = 1728$ such that $\rho_{E_{1},2^{\infty}}(G_{\QQ})$ is conjugate to $\mathcal{N}_{-1,0}(2^{\infty})$. Then the isogeny graph associated to the $\QQ$-class of $E_{1}$ is of $L_{2}(2)$ type and $E_{1}$ is $2$-isogenous to an elliptic curve $E_{2}/\QQ$ with $\textit{j}_{E_{2}} = 1728$. A priori, $\rho_{E_{2},2^{\infty}}(G_{\QQ})$ is conjugate to one of the seven groups, $\mathcal{N}_{-1,0}(2^{\infty})$, $\left\langle G_{2,b}, c_{1} \right\rangle$, $\left\langle G_{2,b}, c_{1}' \right\rangle$, $\left\langle G_{4,c}, c_{1} \right\rangle$, $\left\langle G_{4,c}, c_{1}' \right\rangle$, $\left\langle G_{4,d}, c_{1} \right\rangle$, or $\left\langle G_{4,d}, c_{1}' \right\rangle$. We have eliminated six out of the seven possibilities. Hence, $\rho_{E_{2},2^{\infty}}(G_{\QQ})$ is conjugate to $\mathcal{N}_{-1,0}(2^{\infty})$. Let $E$ be the elliptic curve $y^{2} = x^{3} + 3x$. Then by Example 9.8 in \cite{al-rCMGRs}, $\rho_{E,2^{\infty}}(G_{\QQ})$ is conjugate to $\mathcal{N}_{-1,0}(2^{\infty})$.
        \end{enumerate}
\end{itemize}
\end{proof}

\bibliography{bibliography}

\begin{thebibliography}{1}

\bibitem{al-rCMGRs}
\'{A}lvaro Lozano-Robledo.
\newblock Galois representations attached to elliptic curves with complex
  multiplication.
\newblock {\em arXiv: Number Theory}, 2018.

\bibitem{magma}
Wieb Bosma, John Cannon, and Catherine Playoust.
\newblock The {M}agma algebra system. {I}. {T}he user language.
\newblock {\em J. Symbolic Comput.}, 24(3-4):235--265, 1997.
\newblock Computational algebra and number theory (London, 1993).

\bibitem{gcal-r}
Garen Chiloyan and \'{A}lvaro Lozano-Robledo.
\newblock A classification of isogeny-torsion graphs of {$\mathbb{Q}$}-isogeny
  classes of elliptic curves.
\newblock {\em Trans. London Math. Soc.}, 8(1):1--34, 2021.

\bibitem{kenku}
Monsur~Akangbe Kenku.
\newblock On the number of {$\mathbb{Q}$}-isomorphism classes of elliptic
  curves in each {$\mathbb{Q}$}-isogeny class.
\newblock {\em J. Number Theory}, 15(2):199--202, 1982.

\bibitem{lozano0}
\'{A}lvaro Lozano-Robledo.
\newblock On the field of definition of {$p$}-torsion points on elliptic curves
  over the rationals.
\newblock {\em Math. Ann.}, 357(1):279--305, 2013.

\bibitem{mazur1}
B.~Mazur.
\newblock Rational isogenies of prime degree (with an appendix by {D}.
  {G}oldfeld).
\newblock {\em Invent. Math.}, 44(2):129--162, 1978.

\bibitem{Rouse2021elladicIO}
Jeremy Rouse, Andrew~V. Sutherland, and David Zureick-Brown.
\newblock $\ell$-adic images of galois for elliptic curves over $\mathbb{Q}$
  (and an appendix with john voight).
\newblock {\em Forum of Mathematics, Sigma}, 10:e62, 2022.

\bibitem{Silverman}
J.~H. Silverman.
\newblock {\em The arithmetic of elliptic curves}, volume 106 of {\em Graduate
  Texts in Mathematics}.
\newblock Springer, Dordrecht, second edition, 2009.

\end{thebibliography}
\bibliographystyle{plain}
\end{document}